\documentclass[10pt]{amsart}  
 
\usepackage{graphicx, yfonts, wasysym, appendix, url, verbatim}
\usepackage[margin=4cm]{geometry}
\usepackage{rotating}
\usepackage{amsbsy,enumerate}
\usepackage{graphicx}
\usepackage{ccaption}
\usepackage{xcolor}
\usepackage{times}
\usepackage{wrapfig}
\usepackage{amsmath, amssymb, amsthm} 
\usepackage{algorithm, algorithmic}
\usepackage[sort]{cite}
\usepackage{hyperref}
\usepackage{booktabs}
\usepackage{subcaption}

\newcommand{\vn}[1]{\lVert#1\rVert}
\newcommand{\IP}[2]{\left< #1 , #2 \right>}

\newcommand{\norm}[1]{\left\lVert{#1}\right\rVert}
\newcommand{\abs}[1]{\left\lvert{#1}\right\rvert}

\newcommand{\R}{\ensuremath{\mathbb{R}}}

\newcommand{\N}{\ensuremath{\mathbb{N}}}

\renewcommand{\L}{\ensuremath{\mathcal{L}}}

\allowdisplaybreaks[4]

\newtheorem{thm}{Theorem}[section]

\newtheorem{prop}[thm]{Proposition}
\newtheorem{lem}[thm]{Lemma}

\newtheorem{defn}[thm]{Definition}

\theoremstyle{remark}
\newtheorem*{rmk}{Remark}

\usepackage[frozencache,newfloat]{minted} 
\newcommand{\python}[1]{\mintinline{python}{#1}}
\setminted[python]{ %
    linenos=true,             
    autogobble=true,          
    frame=lines,
    framesep=2mm,
    fontsize=\footnotesize
}
\SetupFloatingEnvironment{listing}{placement=htp}
\usepackage{caption}

\keywords{curve shortening flow, free boundary conditions, biological membranes, geometric
analysis} \subjclass[2000]{53C44\and 58J35}

\title{A curvature flow approach to dorsal closure modelling}
\author{Shuhui He, Ben Whale, Glen Wheeler, Valentina-Mira Wheeler}

\address{Shuhui He \\
           Institute for Mathematics and its Applications \\
           University of Wollongong\\
           Northfields Avenue\\
           Wollongong, NSW, 2522, Australia\\
           email: sh807@uowmail.edu.au
           }
\address{Ben Whale \\
           Institute for Mathematics and its Applications \\
           University of Wollongong\\
           Northfields Avenue\\
           Wollongong, NSW, 2522, Australia\\
           email: bwhale@uowmail.edu.au
           }
\address{Glen Wheeler \\
           Institute for Mathematics and its Applications \\
           University of Wollongong\\
           Northfields Avenue\\
           Wollongong, NSW, 2522, Australia\\
           email: glenw@uow.edu.au
           }
\address{Valentina-Mira Wheeler \\
           Institute for Mathematics and its Applications \\
           University of Wollongong\\
           Northfields Avenue\\
           Wollongong, NSW, 2522, Australia\\
           email: vwheeler@uow.edu.au
           }

\begin{document}

\begin{abstract}

In this paper we propose and study a curvature-based mathematical model for dorsal closure in embryonic drosophila.
Using an analysis that mixes maximum-principle and integral-estimates, we establish global existence and convergence for data that mimics the initial geometry of a dorsal closure event.
Further, we present a numerical approximation scheme for the flow, establishing stability, consistence, and convergence.
We also give sample simulations of the flow with initial configurations that include experimentally observed data.
\end{abstract}

\maketitle

\section{Introduction}
  \label{sec_intro}

Dorsal closure (DC) is a critical morphogenetic stage in \emph{Drosophila} embryogenesis.  
Because DC is a naturally occurring epithelial-closure event, it provides an accessible model for embryonic wound healing and related morphogenetic processes.  
Although DC has been studied extensively\,\cite{jacinto2002dynamic,wood2002wound,hutson2003forces,%
peralta2007upregulation,layton2009drosophila,almeida2011mathematical}, the biological and physical mechanisms that drive it are still not fully understood.

During DC an elliptical gap on the dorsal mid-line of the embryo closes over several hours.  
At the outset a roughly elliptical epithelial opening, covered by the amnioserosa (a transient extra-embryonic tissue), forms on the dorsal surface.  
The interface between epidermis and amnioserosa is called the \emph{leading edge} (LE).  
The major axis of the opening coincides with the anterior–posterior (AP) axis of the embryo.  
Soon after closure begins, the ellipse develops sharp corners at the AP termini and becomes eye-shaped; these corner points are the \emph{canthi}.  
Throughout DC the upper and lower halves of the LE advance towards the AP axis under several forces and ultimately meet to leave a seam (Figure~\ref{DCpic}).

\begin{figure}
  \centering
  \includegraphics[scale=0.35]{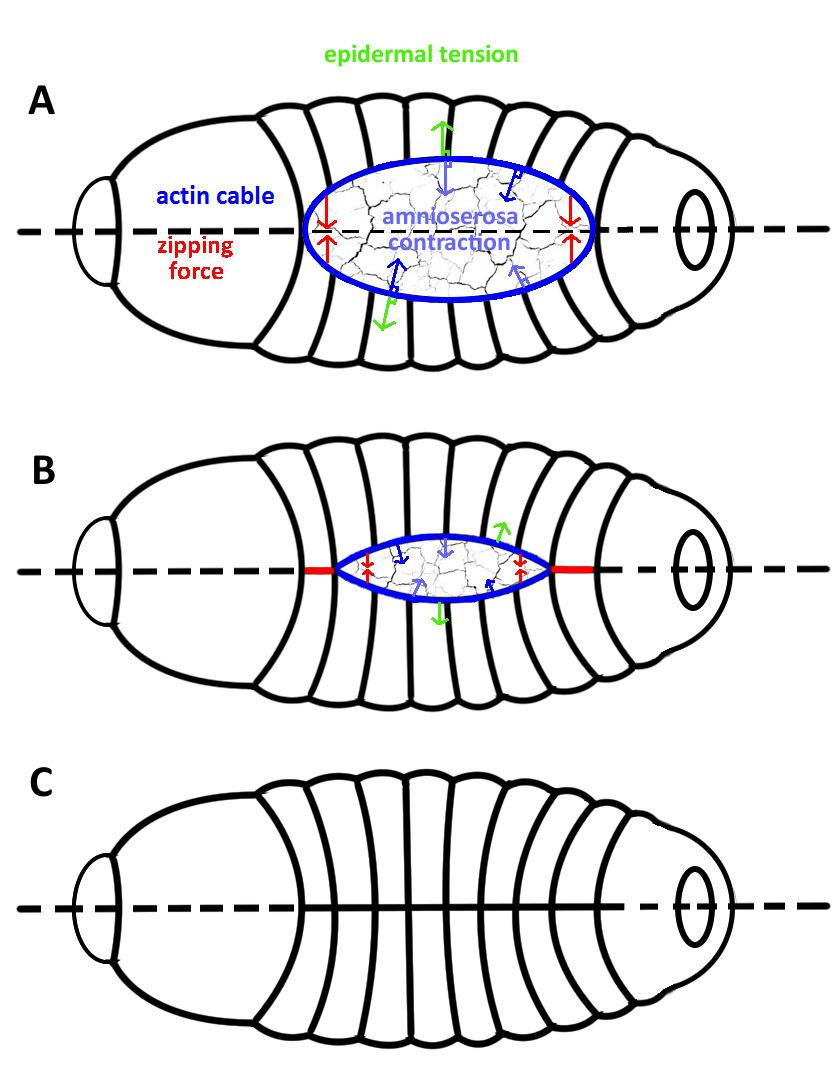}
  \caption{%
    Stages of dorsal closure and the principal forces involved.  The dotted line marks the AP axis.  
    (A) Initial stage: epidermal tension (green), amnioserosa contraction (light blue), zipping at the canthi (red), and global actin-cable contraction (blue) act on the opening.  
    (B) Intermediate stage: an eye-shaped opening develops; the red lines highlight the non-smooth LE geometry at the canthi produced by zipping.  
    (C) Final stage: the opening has closed, leaving a seam along the AP axis.}%
  \label{DCpic}
\end{figure}

Most modelling efforts focus on the evolution of the LE.  
Four principal force categories act on it during DC\,\cite{kiehart2000multiple,almeida2011mathematical,jacinto2002dynamic,wood2002wound}:  
(i) actin-cable contraction, (ii) zipping, (iii) amnioserosa contraction, and (iv) epidermal tension.  
A dense actomyosin cable encircles the opening at the onset of DC; its contraction produces the so-called purse-string effect, shrinking the opening at a rate proportional to curvature.  
\emph{Zipping} arises when micron-scale actin protrusions from opposite sides of the mid-line interlock and pull the tissue edges together, an effect most visible near the canthi.  
Amnioserosa contraction generates an inward normal force, whereas epidermal tension is a comparatively small outward force that resists closure.

A prototypical DC model, proposed by Hutson \emph{et al.}\,\cite{hutson2003forces}, balances these forces on two idealised circular arcs representing the LE and yields a pair of ordinary differential equations (ODEs) for the opening’s length and width.  
Subsequent work refined the empirical zipping term to capture asymmetric scenarios\,\cite{peralta2007upregulation,layton2009drosophila}.  
Almeida \emph{et al.}\,\cite{almeida2011mathematical} advanced the framework with a partial differential equation (PDE) scheme and an alternative treatment of zipping based on chord–arc ratios, but provided no rigorous mathematical analysis.

Building on these ideas, we propose a modified DC model that admits rigorous analysis.  
Idealising the LE as a closed plane curve $\gamma:[0,1]\!\times\![0,T)\to\R^2$, we let it evolve according to
\begin{equation}\label{main}
  \partial_t\gamma(p,t)=\alpha(p,t)\,\vec{k}(p,t)+\beta(p,t)\,\nu(p,t)+\zeta(p,t),
\end{equation}
where $\vec{k}$ is the curvature vector, $\nu$ the inward unit normal, $\alpha$ encodes actin-cable tension, $\beta$ represents the combined effect of amnioserosa contraction and epidermal tension, and $\zeta$ models zipping through the chord–arc ratio.

Two simplifications make \eqref{main} tractable.  
First, because actin-cable tension and amnioserosa contraction act in the same direction and are difficult to distinguish experimentally\,\cite{almeida2011mathematical}, we combine them into a single curvature-driven term.  
Second, instead of modelling the highly localised zipping forces at the canthi directly, we introduce a smooth global term that captures their net effect.  
With these adjustments, the evolution reduces to a curvature-driven (curve-shortening) flow with a restriction term—a class of flows that has been studied extensively\,\cite{ecker1989mce,gage1986heat,grayson1987heat,huisken1986cch,STW2003,CP2009}.

We orient the dorsal opening so that the AP axis coincides with the $x$-axis and the $y$-axis passes through its midpoint.  
Assuming bilateral symmetry, we analyse only the upper‐right quarter of the LE, thereby simplifying the analysis slightly.  
Denoting the curve by $\gamma:[0,1]\!\times\![0,T)\to\R^2$, the specific flow problem is
\begin{equation}\label{DC problem}
\begin{aligned}
  \partial_t\gamma(u,t) &= \kappa(u,t)\,\nu(u,t) + Z^{\perp}(u,t), & &(u,t)\in(0,1)\times[0,T),\\
  \langle\tau,e_2\rangle &= 0, & &u=0,\\
  \gamma(u,t) &= (\rho_0,0), & &u=1,\\
  \gamma(u,0) &= \gamma_0(u), & &u\in[0,1],
\end{aligned}
\end{equation}
with
\begin{equation}\label{Zipping}
  Z(u,t) = -\frac{1}{L(t)}\langle\nu,e_1\rangle\,\langle\nu,e_2\rangle\,e_2,
\end{equation}
where $L(t)$ is the arclength of $\gamma(\cdot,t)$, and $e_1,e_2$ are the standard basis vectors in~$\R^2$.

The first boundary condition enforces orthogonality to the $y$-axis so that $\gamma$ can be reflected smoothly; the second anchors the curve to the $x$-axis.  
The term $Z(u,t)$ ensures that zipping is most significant when the leading edge is about $54$ degrees to the $x$-axis (approximately the angle at the canthi), and that the effect of zipping over the entire curve increases as the length decreases (through the factor~$1/L(t)$).

\subsection{Analysis}

While the flow \eqref{DC problem} is significantly nonlinear, this is confined to lower-order terms and so local well-posedness is standard.

Global analysis however is far more delicate.
Here, we combine in a novel way  maximum-principle and integral-estimate based analyses, yielding finally a powerful convergence result, Theorem \ref{THMmain} below.
Our treatment includes a weak solutions framework, enabling analysis of the smoothing effect.
The main utility of this result is that it confirms the model behaves as expected with initial data that resembles the geometry of a real dorsal closure event.
In formal terms, we prove:

\begin{thm}
Let $\gamma: [0,L(t)]\times[0,T)\to\R^2$ be the solution to \eqref{DC problem} for a given initial curve $\gamma_0:[0,L(0)]\rightarrow\R^2$ of class $W^{2,\infty}$.
  Suppose there exists a $C_G \in (0,1)$ such that $\gamma_0$ is $e_2$-graphical
  with
  \[
    G_{-e_2}(s,0) := \IP{\nu(s,0)}{-e_2} \ge C_G
    \,.
  \]
  Assume that $\gamma_0(s)\cdot e_2 \le \gamma_0(0)\cdot e_2$, $\gamma_0$ has positive curvature at $u=0$, negative curvature at $u=1$, and exactly one point of zero curvature.

  Then:
  \begin{itemize}
  \item The maximal time of existence is infinite;
  \item For any $\varepsilon>0$, the solution is smooth with uniform estimates in $C^{\infty}$ on $(\varepsilon,\infty)$.
  \item As $t\rightarrow\infty$ the solution converges exponentially fast in the
  smooth topology to the horizontal line segment connecting the origin to
  $(\rho_0,0)$.
  \end{itemize}
  \label{THMmain}
  \end{thm}

\begin{rmk}
The initial assumptions on the geometry of $\gamma_0$ mimic the observed DC geometry: the initial curve $\gamma_0$ has positive curvature at $u=0$, negative curvature at $u=1$, and exactly one point of zero curvature.
It is graphical over the $x$-axis and, while not necessarily monotone from the reflection axis, has maximum width there.
\end{rmk}

\begin{rmk}
  If $\gamma_0$ is more regular, uniform higher-order estimates hold on $(0,\infty)$—a prerequisite for proving convergence of the numerical scheme, which requires bounded fourth derivatives up to $t=0$.
\end{rmk}

The analysis-focused sections are as follows:
\begin{itemize}
\item[Section 2] In which we briefly treat short time existence
\item[Section 3] Devoted to fundamental commutation identities, compatibility conditions, and length evolution
\item[Section 4] Specialising to graphical data, we give the main estimates used for global existence and the weak formulation of the flow
\item[Section 5] Leveraging integral estimates to establish convergence (collapse) to a zero-area line segment
\end{itemize}

\begin{figure}
    \centering
    \begin{subfigure}[t]{0.45\textwidth}
      \includegraphics[width=\textwidth]{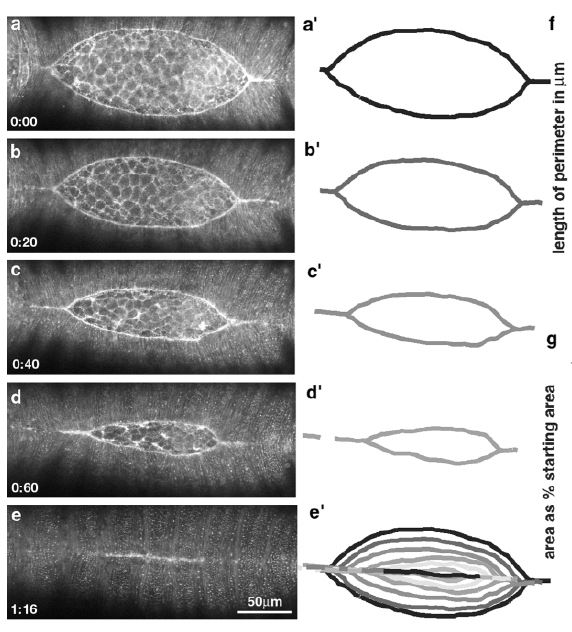}
      \caption{One experimental observation of Drosophila dorsal closure of the
      native type. Picture taken from Figure 3 of \cite{kiehart2000multiple}
      a paper which investigates the 
      forces which contribute to the movement of the leading edge.
We see 7 contours showing the graduate closure of the dorsal
      opening. Each contour is drawn 20 units of time apart from the previous
      one, except the last one is captured 16 time units from the second last
      one.}
      \label{subfigure_experiemnt}
    \end{subfigure}
    \hspace{1em}
    \begin{subfigure}[t]{0.45\textwidth}
      \includegraphics[width=\textwidth]{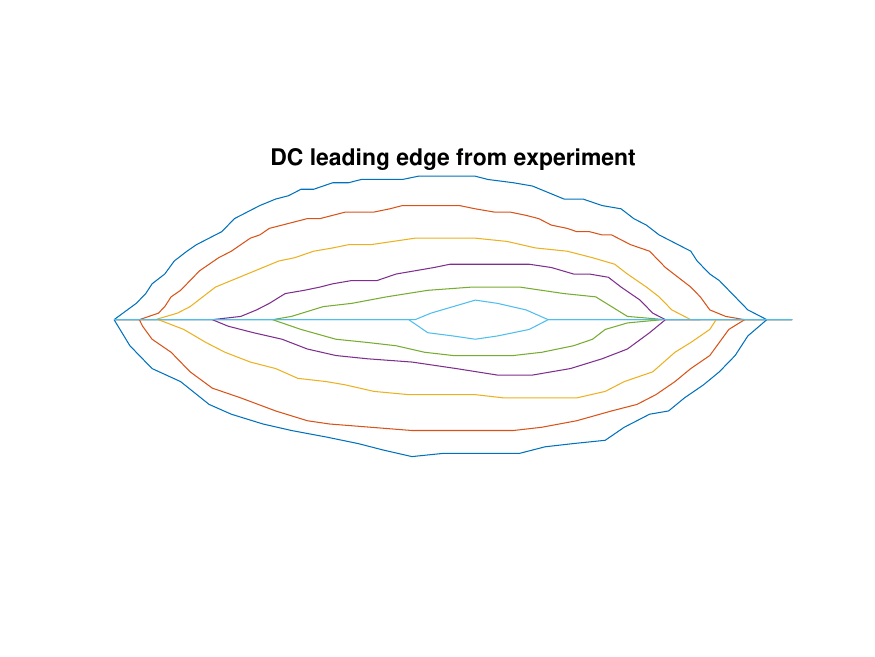}
      \caption{The data extracted from Figure \ref{subfigure_experiemnt} using
      the program ``GetData Graph Digitizer'' and Matlab.}
      \label{subfigure_experiemnt_interp}
    \end{subfigure}
  \end{figure}

\begin{figure}[ht]
  \centering
  \begin{subfigure}[t]{0.31\textwidth}
    \includegraphics[width=\textwidth]{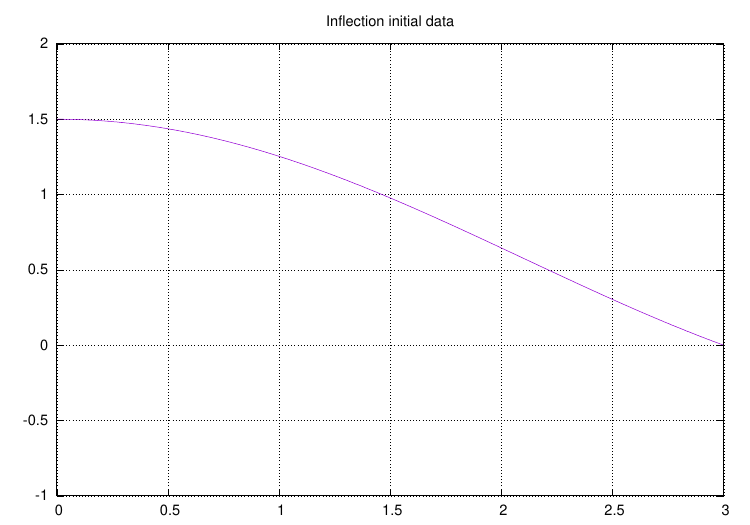}
    \caption{Inflection}\label{fig_initial_inflection}
  \end{subfigure}\hfill
  \begin{subfigure}[t]{0.31\textwidth}
    \includegraphics[width=\textwidth]{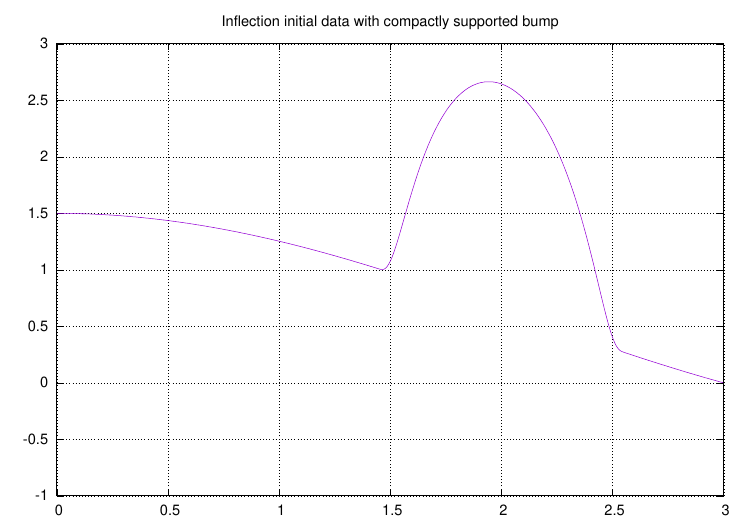}
    \caption{Bump}\label{fig_initial_bump}
  \end{subfigure}\hfill
  \begin{subfigure}[t]{0.31\textwidth}
    \includegraphics[width=\textwidth]{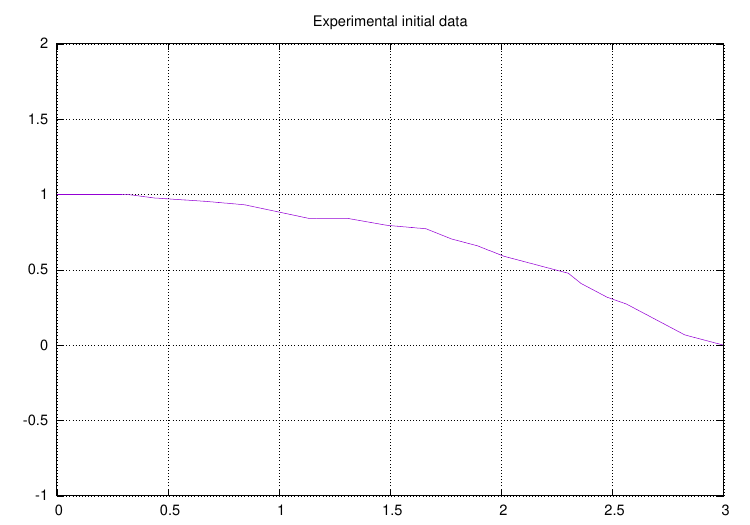}
    \caption{Experimental}\label{fig_initial_experimental}
  \end{subfigure}
  \caption{Initial data used in the simulations.}
  \label{fig:initial-data}
\end{figure}

\begin{figure}[htbp]
  \centering
  \captionsetup[subfigure]{justification=centering}

  \begin{subfigure}[t]{\linewidth}
    \centering
    \includegraphics[width=0.670\linewidth]{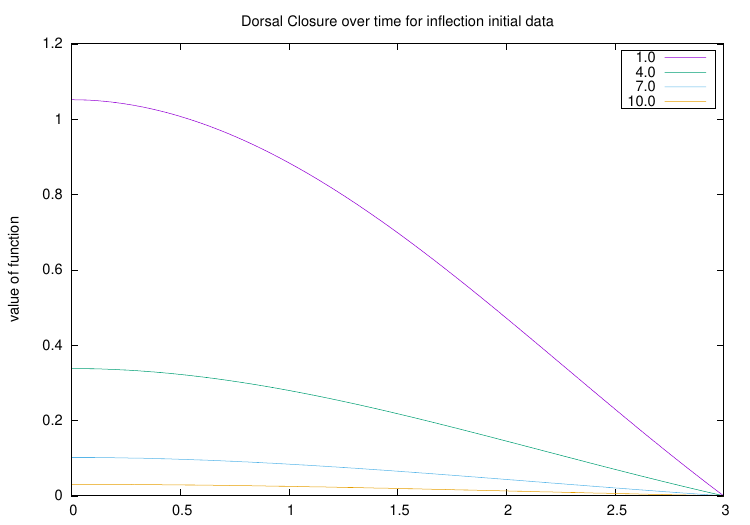}
    \caption{Inflection initial geometry}
    \label{fig:evo-inflection}
  \end{subfigure}\par\medskip

  \begin{subfigure}[t]{\linewidth}
    \centering
    \includegraphics[width=0.670\linewidth]{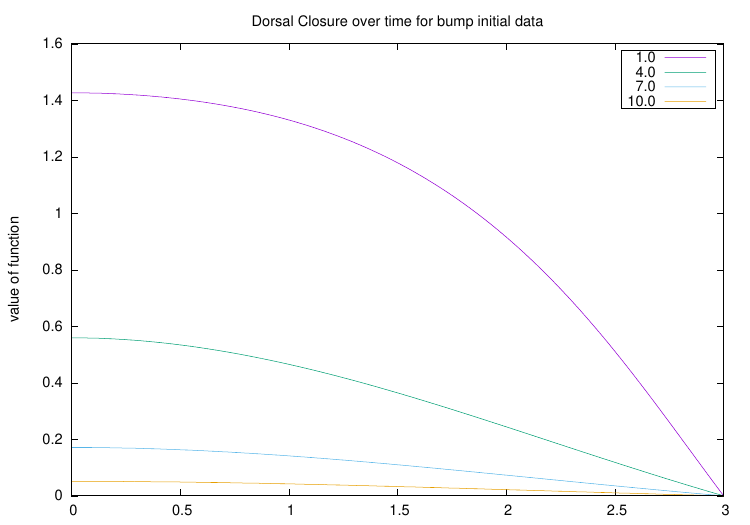}
    \caption{Bump initial geometry}
    \label{fig:evo-bump}
  \end{subfigure}\par\medskip

  \begin{subfigure}[t]{\linewidth}
    \centering
    \includegraphics[width=0.670\linewidth]{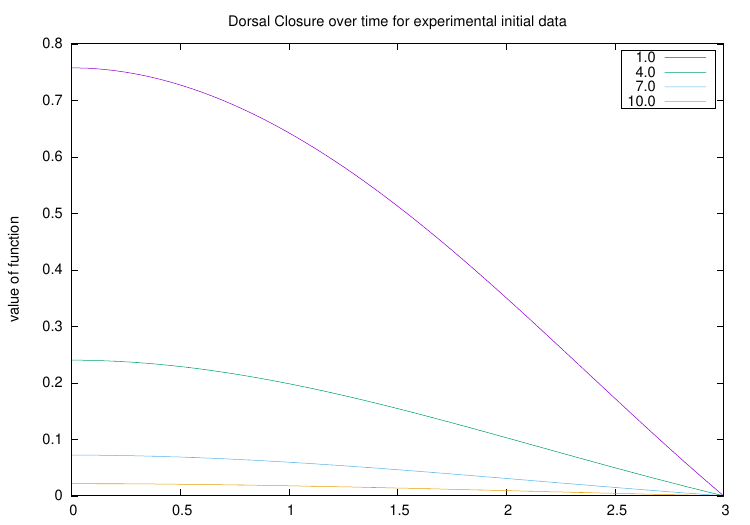}
    \caption{Experimental initial geometry}
    \label{fig:evo-experimental}
  \end{subfigure}

  \caption{Numerical evolution of the leading edge for the three initial geometries.  
           Curves are plotted at equal time intervals.}
  \label{fig:solution-evolution}
\end{figure}

\subsection{Numerical treatment}

The second contribution of this paper is a considered numerical treatment of the flow.
Our plan is to use the graphical formulation (equation \eqref{eqgraph} to come) and prove stability, consistency and convergence results. We prove all results directly as there do not appear to be sufficiently general results for non-local PDE. 

There are two main ideas.
First, that the non-linear terms can be controlled by an $L^\infty$ norm of a derivative of the solution; and, second, we can make use of the existing estimates for the solution of the PDE. The result is that we first prove consistency then convergence and then stability.

\subsection*{Overview of the Numerical Experiments}

Using the described numerical scheme, we performed high-resolution simulations of the flow \eqref{DC problem}.
We worked with the open-source \texttt{COFFEE} infrastructure\,\cite{doulis2019coffee}.  
All runs were performed on a standard desktop (Intel\textsuperscript{\tiny\textregistered} i5-10400 CPU, 16 GB RAM, Ubuntu 20.04) on successively refined one-dimensional grids with $n=20\times2^{\,i}$, $i=0,\dots,10$.  
We tested three representative initial geometries (Figure \ref{fig:initial-data}):

\begin{enumerate}
  \item \textbf{Inflection profile} – a single-inflection sinusoid $g_1$ calibrated to a length–height ratio of $2{:}1$;
  \item \textbf{Bump profile} – a compact perturbation $g_2$ that introduces a localised bulge of controlled height and width;
  \item \textbf{Experimental profile} – a curve extracted from time-lapse microscopy of wild-type \emph{Drosophila} dorsal closure.
\end{enumerate}

Figure \ref{fig:solution-evolution} illustrates the ensuing evolution for each case: the LE shortens monotonically, remains smooth, and converges rapidly to a horizontal segment, in agreement with Theorem \ref{THMmain} and aligning well with observed experimental evidence.  Measured $\log_2 L^\infty$ errors against the finest grid confirm first-order spatial convergence (Table \ref{table_error}, Numerical Section).

There are two sections of the paper devoted to numerics:

\begin{itemize}
\item[Section 6] Consists of the definition and rigorous numerical analysis of our scheme, including stability, consistency and convergence
\item[Section 7] Containing the precise definitions of the initial data simulated and tables of error estimates
\end{itemize}

\section{Short time existence for parametrised curves}
  \label{sec_short_time_param}

  We consider a one-parameter family of curves $\gamma:[0,1]\times[0,T)\to\R^2$ with velocity satisfying
  \begin{equation}
  \label{EQnv}
  \partial_t\gamma = \kappa + Z^\perp
  \end{equation}
  where $\kappa$ is the curvature vector of $\gamma$ and $Z^\perp = \IP{Z}{\nu}\nu$ is the normal part of the zipping force $Z$, given by
  \[
  Z = -\frac{\nu_1\nu_2}{L(t)} e_2\,.
  \]
  Here we have used $L(t)$ to denote the length of $\gamma(\cdot,t)$ and $\nu_1$,
  $\nu_2$ to denote the first and second components of the inward-pointing normal
  vector $\nu$ in the standard basis $\{e_1,e_2\}$.

  Fixing a clockwise parametrisation for the evolving curve $\gamma:[0,1]\times
  [0,T]\to \R^2$, and setting $\gamma(u,t) = (x(u,t),y(u,t))$, we have for the unit tangent, normal and
  curvature vectors the formulae
  \begin{align*}
    \tau &= |\gamma_u|^{-1} (x_u,y_u);\quad
    \nu  = |\gamma_u|^{-1} (y_u,-x_u);\\
    \kappa &= |\gamma_u|^{-2} (x_{uu},y_{uu}) + |\gamma_u|^{-1}(|\gamma_u|^{-1})_u (x_u,y_u)
           = |\gamma_u|^{-2} (x_{uu},y_{uu}) + (|\gamma_u|^{-1})_u \tau = k\nu
           \,,
  \end{align*}
  where
  \[
    k = |\gamma_u|^{-2} \IP{\gamma_{uu}}{\nu}
    \,.
  \]
  Denoting by $s$ the arclength paramter of the curve, so that $\partial_s=\partial_u/|\gamma_u|$,
  the classical Frenet--Serret formulas are
  \begin{equation}\label{6bis}
  \gamma_s = \tau, \qquad \gamma_{ss} = \tau_s = k\nu, \qquad \nu_s = -k\tau
  \,.
  \end{equation}
  The zipping force can be expressed as
  \[
    Z^\perp
    = -\frac{\nu_1\nu_2^2}{L} \nu
    = -\frac{y_ux_u^2}{|\gamma_u|^3L} \nu
    \,.
  \]
  The flow problem \eqref{EQnv} becomes
  \begin{equation}\label{eqpar}
  \begin{cases}
  \partial_t\gamma 
  = \Big(|\gamma_u|^{-2} \IP{\gamma_{uu}}{\nu} -\frac{y_ux_u^2}{|\gamma_u|^3L}
    \Big) \nu\,,
    \\
    \tau(0,t) = e_1\quad\text{ and }\quad\gamma(1,t) = (\rho_0,0)\,.
  \end{cases}
  \end{equation}
  supplemented with boundary conditions $\tau(0,t) = e_1 = (1,0)$ and $\gamma(1,t) = (\rho_0,0)$.

  Using a standard proof for smooth data (see for example Huisken-Polden
  \cite{HP}, Ladyzhenskaya-Uraltseva-Solonnikov \cite{LSU}, Lunardi \cite{Lu} and
  Sharples \cite{JJS}) we have the following local existence result for
  \eqref{eqpar}.

  \begin{thm}\label{STE}
  Let $\gamma_0:[0,1]\to \R^2$ be a smooth map such that $|\gamma_0'(u)|>0$ for
  all $u\in [0,1]$, satisfying $\tau_0(0) = (1,0)$ and $\gamma_0(1) = (\rho_0,0)$
  for a given $\rho_0>0$, as well as the compatibility condition $\kappa_0(1) =
  -Z^\perp_0(1)$.
  There exist $T>0$ and a smooth solution to \eqref{eqpar}, defined on
  $[0,1]\times [0,T)$, such that $\gamma(u,0)=\gamma_0(u)$ for all $u\in [0,1]$.
  \end{thm}
  \begin{proof}
  For smooth data the proof is standard and so we give only a brief outline.
  If we write $\gamma([0,1],t)$ as the normal graph of a function $f(u,t)$ over
  the initial curve $\gamma_0([0,1])$, so that
  \begin{equation}
  \label{formulaforg}
  \gamma(u,t) = \gamma_0(u)+f(u,t)\nu_0(u)\,,
  \end{equation}
  we have
  \begin{align*}
  |\gamma_u| 
   = |\gamma^0_u|\sqrt{ (1 - fk^0)^2 + f_{s_0}^2 }
  \,.
  \end{align*}

  The evolution equation \eqref{eqpar} becomes
  \begin{align}
  f_t
  &= \Big(|\gamma_u|^{-2} \IP{\gamma_{uu}}{\nu} -\frac{y_ux_u^2}{|\gamma_u|^3L}
    \Big) \IP{\nu_0}{\nu}
  \notag
  \\&= \bigg(
    f_{uu}\frac{|\gamma^0_u|(1-k_0f)}{(|f_u|^2 + (1-k_0f)^2|\gamma^0_u|^2)^{\frac32}}
  \notag\\&\qquad
    + f_u\frac{|\gamma_u^0|(-2k_0f_u-k^0_uf - (1-k_0f)(|\gamma_u|^{-1})_u }{(|f_u|^2 + (1-k_0f)^2|\gamma^0_u|^2)^{\frac32}}
  \notag\\&\qquad
    - \frac{
  (y^0_u(1-k_0f) - f_ux^0_u|\gamma^0_u|^{-1})(x^0_u(1-k_0f) + f_uy^0_u|\gamma^0_u|^{-1})^2
  }{L(t)(|f_u|^2 + (1-k_0f)^2|\gamma^0_u|^2)^{\frac32}}
  \notag\\&\qquad
    + \frac{|\gamma^0_u|k_0(1-k_0f)^2}{(|f_u|^2 + (1-k_0f)^2|\gamma^0_u|^2)^{\frac32}}
    \bigg)
    \frac{(1-k_0f)|\gamma^0_u|^2}{|\gamma^0_u|\,\sqrt{|f_u|^2 + (1-k_0f)^2|\gamma^0_u|^2}}
  \,.
  \label{quasi}
  \end{align}
  supplemented with boundary conditions
  \begin{align*}
  f_u(0,t) = 0\quad\text{and}\quad
  f(1,t) = 0
  \,.
  \end{align*}
  Although quite messy, \eqref{quasi} is a uniformly parabolic quasilinear
  initial boundary value problem, and finding a solution for a short time with smooth initial data
  with an arbitrary positive smooth function of time in place of $L$ follows
  by using any of the aforementioned references: \cite{HP,LSU,Lu,JJS}.
  The conclusion for the particular choice of function
  \[
  L[f] = \int_0^1 \sqrt{|f_u|^2 + (1-k_0f)^2|\gamma^0_u|^2}\,du
  \]
  follows by using a fixed point argument exactly as in e.g. McCoy \cite{mccoy} (see also \cite{gwthesis}).

  Once we have solved for $f$, then we can use the formula \eqref{formulaforg} to reconstruct a solution to \eqref{eqpar}.
  \end{proof}

\section{Commutators, further boundary conditions, and evolving length}
  \label{sec_commutators_boundaries_length}

  Let us reparametrise the family of curves by its arc-length and denote the
  arc-length parameter with $s$. The family of curves $\gamma(s,t)$ is now
  defined on the domain $[0,L(t)]\times[0,T)$ with satisfying \eqref{eqpar} and boundary conditions
  $\tau(0,t) = e_1$, $\gamma(L(t),t) = (\rho_0,0)$.

  We shall now derive the commutator for the operators $\partial_s$ and $\partial_t$\,.
  \begin{lem}\label{Lem DC commutator}
  Let $\gamma: [0,L(t)]\times[0,T)\to\R^2$ be the solution to \eqref{eqpar} given
  by Theorem \ref{STE} for a given initial curve
  $\gamma_0:[0,L(0)]\rightarrow\R^2$.
  Let
  \[
  z = -\frac{\nu_1\nu_2^2}{L(t)}\,.
  \]
  The differential operators with respect to arc-length and time satisfy
  \[
  \partial_t\partial_s=\partial_s\partial_t + \left(k^2 + kz\right)\partial_s\,.
  \]
  \end{lem}
  \begin{proof}
  With the Frenet-Serret equations, we compute
  \begin{align*}
  \partial_t|\gamma_u|^2 
    &= 2\IP{\gamma_{ut}}{\gamma_u} = 2 |\gamma_u|^2\IP{(\kappa + Z^\perp )_s}{\gamma_s}
    \\&
    = 2 |\gamma_u|^2\left(-k^2 - kz \right).
  \end{align*}
  Hence
  \begin{align*}
  \partial_t\partial_s 
    & = \partial_t\big( |\gamma_u|^{-1}\partial_u\big) = |\gamma_u|^{-1}\partial_u\partial_t - |\gamma_u|^{-2}|\gamma_u|_t\, \partial_u
    \\
    & =\partial_s\partial_t + \left(k^2 + kz \right)\partial_s\,.
  \end{align*}
  \end{proof}

  \begin{lem}\label{Lem tau nu evolution}
  Let $\gamma: [0,L(t)]\times[0,T)\to\R^2$ be the solution to \eqref{eqpar} given
  by Theorem \ref{STE} for a given initial curve
  $\gamma_0:[0,L(0)]\rightarrow\R^2$.
  We have the following evolution equations
  \begin{align*}
  \tau_t &= \big(k_s + z_s \big)\nu
  \,.
    \\
  \nu_t &= -\big(k_s + z_s \big)\tau
  \,.
  \end{align*}
  \end{lem}
  \begin{proof}
  Applying the previous lemma to the flow equation, we calculate
  \begin{align*}
  \tau_t &= (\gamma_s)_t = \gamma_{ts} + \left(k^2 + kz \right)\tau
    \\&= k_s\nu - k^2\tau + z_s\nu - kz\tau + \left(k^2 + kz \right)\tau
    \\&= \big(k_s + z_s  \big)\nu
  \,.
  \end{align*} 
  As $\gamma(s,t)$ are plane curves, we have $\nu_t = \lambda\tau$ for a scalar $\lambda$. Thus by differentiating $\IP{\tau}{\nu}=0$, we immediately obtain 
  \begin{align*}
  \nu_t 
  &= -\big(k_s + z_s \big)\tau
  \,.
  \end{align*}
  \end{proof}

  For convenience, we explicitly state the expression for derivatives of the scalar zipping factor $z$.

  %
  %
  %

  \begin{lem}
  Let $\gamma: [0,L(t)]\times[0,T)\to\R^2$ be the solution to \eqref{eqpar} given
  by Theorem \ref{STE} for a given initial curve
  $\gamma_0:[0,L(0)]\rightarrow\R^2$.
  We have
  \begin{align*}
  z_s &= \frac{k}L\Big(
    -3\nu_2^3 + 2\nu_2
    \Big)
  \\
  z_{ss} &= \frac{k_s}L\Big(
    -3\nu_2^3 + 2\nu_2
    \Big)
   + \frac{k^2}L\Big(
    -9\nu_1^3 + 7\nu_1
    \Big)
  \,.
  \end{align*}
  \label{Z derivatives}
  \end{lem}

  The boundary conditions give rise to the following.

  \begin{lem}
  \label{LMbcs}
  Let $\gamma: [0,L(t)]\times[0,T)\to\R^2$ be the solution to \eqref{eqpar} given
  by Theorem \ref{STE} for a given initial curve
  $\gamma_0:[0,L(0)]\rightarrow\R^2$.
  We have
  \[
  k_s(0,t) = -\frac{k(0,t)}{L(t)}\,,\quad\text{and}\quad
  k(L(t),t) = \frac{\nu_1(L(t),t)\nu_2^2(L(t),t)}{L(t)}
  \,.
  \]
  \end{lem}
  \begin{proof}
  Each relation follows from differentiating the boundary conditions in time.
  For the boundary point at zero, we have (Lemma \ref{Lem tau nu evolution})
  \[
  0 = \tau_t(0,t) = (k_s(0,t) + z_s(0,t))\nu(0,t)
  \]
  which implies
  \[
    k_s(0,t) = -z_s(0,t)\,.
  \]
  We may further simplify this by noting that
  \[
  %
  %
  z_s(0,t) = \frac{k(0,t)}{L(t)}
  \]
  giving the first relation.
  For the second, the Dirichlet condition gives $\gamma_t(L(t),t) = 0$ which is
  \[
  0 = k(L(t),t) + z(L(t),t)
  \,,
  \]
  as required.
  \end{proof}

  We also calculate the evolution of length.

  \begin{lem}
    Let $\gamma: [0,L(t)]\times[0,T)\to\R^2$ be the solution to \eqref{eqpar} given
    by Theorem \ref{STE} for a given initial curve
    $\gamma_0:[0,L(0)]\rightarrow\R^2$.
    We have
    \[
    L'(t) 	\le -\int_\gamma k^2ds
    \]
    and if $\gamma_2(0,t) > 0$ for $t\in[0,T^*)$, then
    \[
    L'(t)	\le -\frac{4\rho_0\gamma_2(0,t)^2}{(\rho^2_0 + \gamma_2(0,t)^2)^2}
      \,,\quad\text{ for $t\in[0,T^*)$}
    \,.
    \]
    \label{lmlengthevo}
  \end{lem}
  \begin{proof}
    The arclength element evolves by
    \[
      ds_t = (-k^2-kz)\,ds\,.
    \]
    Therefore
    \begin{align}
      \label{eq_lengthevo}
      L'(t) &= \int_\gamma -k^2ds + \frac1L \int_\gamma k\nu_1\nu_2^2ds
      \,.
    \end{align}
    Now, note that
    \begin{align*}
    \frac1L \int_\gamma k\nu_1\nu_2^2ds
     &= \frac1L \int_\gamma k\tau_2\nu_2^2ds
    \\
     &= -\frac1L \int_\gamma (\nu_2)_s\nu_2^2ds
    \\
     &= -\frac1{3L} \int_\gamma (\nu_2^3)_s\,ds
    \\
     &= -\frac1{3L} \Big[
      (\nu_2^3)(L(t),t)
      - (\nu_2^3)(0,t)
        \Big]
    \,.
    \end{align*}
    Now the second component of the unit normal takes the value $-1$ at zero, and this is the smallest possible value it can take. This means that the term in the brackets is non-negative, which yields finally the estimate
    \[
    \frac1L \int_\gamma k\nu_1\nu_2^2ds \le 0
    \,.
    \]
    We therefore have
    \begin{equation}
    \label{EQlengthevo}
    L'(t) \le -\int_\gamma k^2ds
    \,.
    \end{equation}
    If $\gamma(0,t) > 0$, then the curve with the smallest $||k||_2^2$ that is smooth and connects the boundary points together is a circular arc with radius
    \[
      r = \frac{\rho_0^2 + \gamma_2(0,t)^2}{2\gamma_2(0,t)}
    \,.
    \]
    The arclength of the circular arc is $r \arcsin( \rho_0/r )$, so we find the estimate
    \[
    \int_\gamma k^2\,ds \ge \frac{1}{r^2}r \arcsin( \rho_0/r )
      = \frac1{\rho_0}\frac{\rho_0}{r} \arcsin(  \rho_0/r )
    \,.
    \]
    Clearly (from for instance the Taylor series for $x\mapsto x\arcsin(x)$) we have
    \[
    \int_\gamma k^2\,ds
      \ge \frac1{\rho_0}\frac{\rho^2_0}{r^2}
      = \frac{4\rho_0\gamma_2(0,t)^2}{(\rho^2_0 + \gamma_2(0,t)^2)^2}
    \,.
    \]
    Combining this estimate with \eqref{EQlengthevo} finishes the proof.
  \end{proof}


\section{Graphical data}
  \label{sec_graphical}

  In the case where the initial data is graphical, which is the physically interesting setting, we are able to obtain finer results.
  First, the evolution equation preserves graphicality, as per the following result that uses the maximum principle.

  \begin{thm}[Graph preservation]
    Let $\gamma: [0,L(t)]\times[0,T)\to\R^2$ be the solution to \eqref{eqpar} given
    by Theorem \ref{STE} for a given initial curve
    $\gamma_0:[0,L(0)]\rightarrow\R^2$.
    For a fixed vector $V\in\R^2$, set
    \[
    G_{V}(s,t) := \IP{\nu(s,t)}{V}
    \,.
    \]
    If $G_{-e_2}(s,0) > 0$ then $G_{-e_2}(s,t) > 0$.
    \label{graphpreservation}
  \end{thm}
  \begin{proof}
    We compute
    \[
    \partial_s G_V = \IP{\nu_s}{V} = -\IP{k\tau}{V},
    \]
    and
    \[
    \partial_{ss}G_V = -k_s\IP{\tau}{V}-k^2\IP{\nu}{V}.
    \]
    Recalling Lemma~\ref{Lem DC commutator} and Lemma~\ref{Lem tau nu evolution}, we see that
    \begin{align*}
    \partial_t G_V 
      &= \IP{\nu_t}{V}
      = -\big(k_s + z_s\big)\IP{\tau}{V}
      \\
      &= -k_s\IP{\tau}{V}
       - \IP{\tau}{V} \frac{k}L\Big( -3\nu_2^3 + 2\nu_2 \Big)
      \\
      &= -k_s\IP{\tau}{V} + \partial_sG_V \frac{1}L\Big( -3\nu_2^3 + 2\nu_2 \Big)
    \,.
    \end{align*}
    Thus
    \begin{align*}
    \partial_tG_V 
      &= \partial_{ss}^2G_V 
        + k^2G_V
        + \partial_sG_V \frac{1}L\Big( -3\nu_2^3 + 2\nu_2 \Big)
    \,.
    \end{align*}
    By defining the operator $\L$ to be
    \begin{align*}
    \L\phi = \left(\partial_t - \frac{1}L\Big( -3\nu_2^3 + 2\nu_2 \Big)\partial_s - \partial_{ss}\right)\phi,
    \end{align*}
    we can write
    \begin{align*}
    \L G_V = k^2G_V\,.
    \end{align*}
    Hence, by the minimum principle,
    \[
    \min\limits_{[0,L(t)]\times[0,T)} G_{V}(s,t) \ge \min\big\{\min G_{V}(s,0), \min G_{V}(0,t), \min G_{V}(L(t),t)\big\}.
    \]
    Now let us consider $V = -e_2$.
    First, $G_{-e_2}(0,t) = 1$, which is the largest value $G$ can possibly take.

    On the Dirichlet boundary, we apply the Hopf Lemma.
    We find
    \[
    0 > (\partial_s G_{-e_2})(L(t),t)
      = k(L(t),t)\tau_2(L(t),t)
    \,.
    \]
    From Lemma \ref{LMbcs},
    \[
    k(L(t),t)\tau_2(L(t),t) = \frac{\nu_1\tau_2\nu_2^2}{L(t)} = \frac{\nu_1^2\nu_2^2}{L(t)}
    \]
    and so
    \[
    0 > (\partial_s G_{-e_2})(L(t),t)
      = k(L(t),t)\tau_2(L(t),t)
     \ge 0\,,
    \]
    a contradiction.
    We conclude that no new minima for $G_{-e_2}$ can be attained on either boundary, and so
    \[
    \min\limits_{[0,L(t)]\times[0,T)} G_{V}(s,t) \ge \min\big\{\min G_{V}(s,0)\big\}\,,
    \]
    implying the result.
  \end{proof}

  \begin{rmk}
    Along an $e_2$-graphical solution, we have $\tau_2(L(t),t) \le 0$ and so the
    curvature at the right-hand boundary $k(L(t),t)$ is \emph{non-positive}.
  \end{rmk}


  \begin{lem}
  Let $\gamma: [0,L(t)]\times[0,T)\to\R^2$ be the solution to \eqref{eqpar} given
  by Theorem \ref{STE} for a given initial curve
  $\gamma_0:[0,L(0)]\rightarrow\R^2$.
  Suppose
  \[
  \gamma_2(s,0) > 0\,,\quad\text{for all}\ s\in[0,L(0))
  \,.
  \]
  Then
  \[
  \gamma_2(s,t) > 0\,,\quad\text{for all}\ (s,t)\in[0,L(t))\times[0,T)
  \,.
  \]
  \label{DC_1st quadrant graph}
  \end{lem}
  \begin{proof}
  We first prove the inequality holds in the interior. From \eqref{EQnv} we calculate
  \begin{align*}
  \left(\partial_t-\partial_{ss}\right)\gamma_2 = -\frac{\nu_1\nu_2^3}{L} = -\frac{\tau_2\nu_2^3}{L} = -\frac{\nu_2^3}{L}\partial_s\gamma_2
  \,.
  \end{align*}
  That is,
  \[
  \left(\partial_t+\frac{\nu_2^3}{L}\partial_s-\partial_{ss}\right)\gamma_2 = 0
  \,.
  \]
  The maximum and minimum principle apply to this operator and therefore also to $\gamma_2$.
  In particular, we have
  \[
  \min\limits_{[0,L(t)]\times[0,T)} \gamma_2(s,t) \ge \min\big\{\min \gamma_2(s,0), \min \gamma_2(0,t), \min \gamma_2(L(t),t)\big\}.
  \]
  If there is a new minimum at the left hand boundary point, the Hopf Lemma implies
  \[
  0 > -(\partial_s\gamma_2)(0,t) = -\tau_2(0,t) = 0
  \]
  which is a contradiction.
  Therefore the minimum of $\gamma_2$ (under the given initial condition) is
  attained at the Dirichlet boundary point, and furthermore, there are no new
  local minima attained by $\gamma_2$ on the interior or at the Neumann boundary.
  \end{proof}

  \begin{rmk}
  If the initial curve has only one (interior) inflection point, then the
  Sturmian theorem can be used to preserve $k(0,t) > 0$ and $k(L(t),t) < 0$ for
  $t\in[0,T)$.
  Then, because
  \[
  (\partial_t\gamma_2)(0,t)
    = -\frac{\nu_2^3}{L}\tau_2 + k\nu_2 = -k(0,t)
  \]
  we have strict monotonicity of the height at the Neumann boundary.
  \end{rmk}

  \begin{prop}[Enclosed area decay]
  Let $\gamma: [0,L(t)]\times[0,T)\to\R^2$ be the solution to \eqref{eqpar} given
  by Theorem \ref{STE} for a given initial curve
  $\gamma_0:[0,L(0)]\rightarrow\R^2$.
  Suppose $\gamma_0$ is $e_2$-graphical with
  \[
    G_{-e_2}(s,0) \ge C_G
  \]
  where $C_G \in (0,1)$.
  Assume further that $\gamma_2(s,0) \le \gamma_2(0,0)$

  Then the area bounded by the curve and the positive axes decays exponentially fast, satisfying
  \[
  A(t) < A(0)e^{-\frac{C_G^2}{\rho_0L(0)}t}
  \,,
  \]
  for all $t\in(0,T)$.
  \label{DC_area evolution}
  \end{prop}
  \begin{proof}
  Calculating,
  \begin{align*}
  \frac{d}{dt}\int_\gamma \gamma_2\,ds
    &= \int_\gamma
      -\frac{\nu_2^3}{L}\partial_s\gamma_2+\partial_{ss}\gamma_2
     \,ds
     + \int_\gamma -k^2\gamma_2 - kz\gamma_2\,ds
    \\
    &= -\int_\gamma
      3k\gamma_2\frac{\tau_2\nu_2^2}{L}
     \,ds
    + \Big[
      -\frac{\nu_2^3\gamma_2}{L} + \tau_2
    \Big]_0^L
     + \int_\gamma -k^2\gamma_2 + k\gamma_2\frac{\nu_1\nu_2^2}{L}\,ds
    \\
    &=
     -\int_\gamma k^2\gamma_2\,ds
     - \frac{2}{L}\int k\gamma_2\nu_1\nu_2^2\,ds
    + \Big[
      -\frac{\nu_2^3\gamma_2}{L} + \tau_2
    \Big]_0^L
    \\
    &=
     -\int_\gamma k^2\gamma_2\,ds
     - \frac{2}{L}\int k\gamma_2\nu_1\nu_2^2\,ds
    + \tau_2(L,t)
    - \frac{\gamma_2(0,t)}{L}
    \,.
  \end{align*}
  Now we use the initial conditions.
  Note that the maximum principle (see Lemma \ref{DC_1st quadrant graph}) implies that the initial condition $\gamma_2(s,0) \le \gamma_2(0,0)$ is preserved, that is,
  \[
    \gamma_2(s,t) \le \gamma_2(0,t)
  \,.
  \]
  Theorem \ref{graphpreservation} preserves graphicality, and in particular we have
  \[
    \nu_1^2 \le 1 - C_G^2
    \,.
  \]
  We also trivially have $\nu_2^2 < 1$ (note the inequality is strict for $s\in(0,L(t)]$, by the maximum principle applied to $\gamma_2$ (see Lemma 
  \ref{DC_1st quadrant graph})).
  Therefore
  \[
     -\int_\gamma k^2\gamma_2\,ds
     - \frac{2}{L}\int k\gamma_2\nu_1\nu_2^2\,ds
     \le \frac{1}{L^2}\int \nu_1^2\nu_2^4 \gamma_2\,ds
     < (1-C_G^2) \frac{\gamma_2(0,t)}{L}
  \,.
  \]
    Finally, note that Lemma \ref{graphpreservation} also implies $\tau_2(L(t),t) \le 0$.
  Combining this with the evolution of $A$ above, we find
  \begin{align*}
  \frac{d}{dt}\int_\gamma \gamma_2\,ds
    &< (1-C_G^2) \frac{\gamma_2(0,t)}{L(t)}
    - \frac{\gamma_2(0,t)}{L(t)}
    \\
    &= -C_G^2\frac{\gamma_2(0,t)}{L(t)}
  \,.
  \end{align*}
  Since $A(t) \le \gamma_2(0,t)\rho_0$ (this is because $0 \le \gamma_2(s,t) \le \gamma_2(0,t)$), and $\rho_0 \le L(t) \le L(0)$ we find
  \[
  A'(t) < -\frac{C_G^2}{\rho_0L(0)}A(t)
  \,.
  \] 
  This yields the stated decay estimate.
  \end{proof}

  We take now the parametrisation
  \[
  \gamma(u,t)= (u, h(u,t))
  \,,
  \]
  with $u\in [0,\rho_0]$ and $h\in C^\infty([0,\rho_0])$, with the following boundary conditions:
  \[
  h_u(0,t) = 0\quad \text{and} \quad
  h(\rho_0,t) = 0
  \,,
  \]
  for all $t\in [0,T)$.
  In this parametrisation, the evolution equation \eqref{eqpar} becomes
  \begin{equation}
  \label{eqgraph}
  h_t =  \frac{h_{uu}}{1+h_{u}^2} + \frac{1}{L[h]}\frac{h_u}{1+h_{u}^2}
  \,.
  \end{equation}
  We say that \emph{$\gamma$ solves \eqref{eqgraph}} iff $\gamma(u,t)= (u,h(u,t))$, where the function $h$ solves \eqref{eqgraph}.

  \begin{thm}\label{LTE}
  Let $h_0\in C^{\infty}([0,\rho_0])$, with $(h_0)_u(0)=0$ and $h_0(\rho_0)=0$.
  Then, the maximal time of existence is infinite, and equation
  \eqref{eqgraph} admits a smooth solution $h\in C^\infty([0,\rho_0]\times
  [0,\infty))$. 
  \end{thm}
  \begin{proof}
  By standard parabolic regularity theory \cite{Li,LSU}, if the maximal time of
  existence is finite, then the gradient of $h$ must blow up as $t\rightarrow T$.
  However Theorem \ref{graphpreservation} ensures that $h_u$ is uniformly bounded
  for all $t$, and so we conclude that $T=\infty$.
  \end{proof}

  We now prove the following compactness result.

  \begin{prop}\label{provo}
    Let $T>0$ be arbitrary.
    Let $h_0\in W^{1,\infty}([0,\rho_0])$, with $(h_0)_u(0)=0$ (weakly) and $h_0(\rho_0)=0$.
    Take $h_n\in C^\infty([0,\rho_0]\times [0,T])$ to be the solutions of
    \eqref{eqgraph} given by Theorem \ref{LTE} with $h_n(\cdot,0)\rightarrow
    h_0$ in $H^1$.
    Then there exists $h\in H^1([0,T],L^2([0,\rho_0]))\cap L^\infty([0,T],W^{1,\infty}([0,\rho_0]))$
    such that $h_n\to h$, up to a subsequence, uniformly on $[0,\rho_0]\times [0,T]$.
  \end{prop}
  \begin{proof}
  By Theorem \ref{LTE} for any $T>0$, the solutions $h_n$ are uniformly bounded in $L^\infty([0,T],W^{1,\infty}([0,\rho_0]))$. 
  In order to prove the result we also need a uniform bound for $h_n$ in $H^1([0,T],L^2([0,\rho_0]))$.
  For this it remains to estimate
  \[
  \int_0^T \int_0^{\rho_0} |(h_n)_t|^2\,du\,dt
  \,.
  \]
  We find
  \begin{align*}
  \int_0^{\rho_0} |(h_n)_t|^2\,du
    &= \int_0^{\rho_0} \Big|\frac{(h_n)_{uu}}{1+(h_n)_{u}^2} + \frac{1}{L[(h_n)_]}\frac{(h_n)_u}{1+(h_n)_{u}^2} \Big|^2\,du
    \\
    &\le 2\int_0^{\rho_0} \frac{(h_n)_{uu}^2}{(1+(h_n)_u^2)^2}\,du
     + \frac{2}{L^2[(h_n)]}\int_0^{\rho_0}\frac{(h_n)_u^2}{(1+(h_n)_{u}^2)^2}\,du
    \\
    &\le 2\int_0^{\rho_0} \frac{(h_n)_{uu}^2}{(1+(h_n)_u^2)^2}\,du
     + \frac{2}{\rho_0^2}
  \,.
  \end{align*}
  Preservation of graphicality (Theorem \ref{graphpreservation}) implies the gradient bound
  \[
    (h_n)_u^2 \le  \frac{1-C_G^2}{C_G^2}
  \]
  and in particular
  \[
  \frac{1}{\sqrt{1+(h_n)_u^2}} \ge C_G
  \,.
  \]
  (Recall that $C_G$ depends only on $\vn{(h_0)_u}_{L^\infty}$.)
  This allows us to estimate
  \begin{align*}
  \int_\gamma k^2\,ds
    &=   \int_0^{\rho_0} \frac{(h_n)_{uu}^2}{(1+(h_n)_u^2)^3}(1+(h_n)_u^2)^\frac12\,du
  \\&=   \int_0^{\rho_0} \frac{(h_n)_{uu}^2}{(1+(h_n)_u^2)^2}(1+(h_n)_u^2)^{-\frac12}\,du
  \\&\ge C_G\int_0^{\rho_0} \frac{(h_n)_{uu}^2}{(1+(h_n)_u^2)^2}\,du
  \,.
  \end{align*}
  We then integrate Lemma \ref{lmlengthevo} to find
  \begin{align*}
    2\int_0^T \int_0^{\rho_0} \frac{(h_n)_{uu}^2}{(1+(h_n)_u^2)^2}\,du\,dt
    &\le
    2C_G^{-1}\int_0^T \int_\gamma k^2\,ds\,dt
    \\
    &= 2C_G^{-1}\int_0^T -L'(t)\,dt
    = 2C_G^{-1}(L(0) - L(T))
    \le 2C_G^{-1}L(0)
  \end{align*}
  which depends only on $\vn{(h_0)_u}_{L^\infty}$.

  We therefore find
  \[
  \int_0^T \int_0^{\rho_0} |(h_n)_t|^2\,du\,dt
    \le 2C_G^{-1}L(0) + \frac{2T}{\rho_0^2}
  \]
  which is an estimate that depends only on $\vn{(h_0)_u}_{L^\infty}$ and $T$.

  It then follows that the sequence $h_n$ converges, up to a subsequence as $n\to\infty$, to a limit function 
  $h$ in the weak topology of $H^1([0,T],L^2([0,{\rho_0}]))\cap L^\infty([0,T],H^1([0,{\rho_0}]))$.
  Due to the continuous embedding (this depends strongly on the fact that we are in one dimension)
  \[
  H^1([0,T],L^2([0,{\rho_0}]))\cap L^\infty([0,T],H^1([0,{\rho_0}]))\hookrightarrow  
  C^{\frac 1 2,\frac 1 4}([0,{\rho_0}]\times [0,T])\,,
  \]
  the convergence is uniform.
  \end{proof}

  We now define the notion of weak solution. 
   
  \begin{defn}\label{defw}
    We say that a function $h\in H^1([0,T],L^2([0,{\rho_0}]))\cap L^\infty([0,T],H^1([0,{\rho_0}]))$ 
    is a weak solution of \eqref{eqgraph} if 
    \begin{equation}\label{eqdefw}
      \int_{[0,{\rho_0}]\times [0,T]} \left( h_t \varphi 
      + \arctan(h_u)\varphi_u - \frac1{L[h]}\frac{h_u}{1+h_u^2}\,\varphi\right) \,du\,dt = 0
    \end{equation}
    for all test functions $\varphi \in C^1_c([0,{\rho_0}]\times(0,T))$, with boundary condition $h({\rho_0},t) = 0$.
  \end{defn}

  We have the following existence theorem  for weak solution to  \eqref{eqgraph}.

  \begin{thm}
    \label{thmw2inf}
    Let $h_0\in W^{2,\infty}([0,{\rho_0}])$ with 
    $(h_0)_u(0) = 0$ and $h_0({\rho_0}) = 0$.
    Then, there exists $T>0$ depending only on $h_0$
    such that equation \eqref{eqgraph} admits a weak solution 
    \[
      h\in 
      W^{1,\infty}([0,T],L^\infty([0,{\rho_0}]))
      \cap 
      L^\infty([0,T],W^{2,\infty}([0,{\rho_0}]))
      \,.
    \]
  \end{thm}
  \begin{proof}
  By Proposition \ref{provo} there exist $T>0,$ depending only on $\left\|h_{0}\right\|_{H^{1}}$ and smooth solutions
  $h_{n}$ of \eqref{eqgraph} which converge, up to a subsequence, uniformly to a limit function $h$ in
  the weak topology of $H^{1}\left([0, T], L^{2}([0,1])\right) \cap
  L^{\infty}\left([0, T], H^{1}([0,1])\right)$.
  We wish to prove that $u$ is a weak solution of \eqref{eqgraph}. The primary difficulty is in showing that 
  ${u_n}_x$ converges to $u_x$ almost everywhere, so that we can pass to the limit in \eqref{eqdefw}.

  We compute (in the following computation we use $\tilde{h} = h_n$)
  \begin{align}\label{uut}
  \nonumber 
  \partial_t \frac{\tilde h_t^2}{2} &= \tilde h_t\, \tilde h_{tt} = \tilde h_t\left( \frac{\tilde h_{uu}}{1+\tilde h_u^2} + \frac1{L[\tilde h]}\frac{\tilde h_u}{1+\tilde h_u^2}\right)_t 
  \\
  &= \frac{\tilde h_t\,\tilde h_{tuu}}{1+\tilde h_u^2} - 2 \frac{\tilde h_u\,\tilde h_{uu}}{(1+\tilde h_u^2)^2}\left( \frac{\tilde h_t^2}{2} \right)_u
  + \frac1{L[\tilde h]}\bigg(
    \frac{1}{1+\tilde h_u^2}
    - 2\frac{\tilde h_u}{(1+\tilde h_u^2)^2}
    \bigg)\left( \frac{\tilde h_t^2}{2} \right)_u
  \\&\quad\nonumber
    - \tilde h_t\frac{L'[\tilde h]}{L^2[\tilde h]}\frac{\tilde h_u}{1+\tilde h_u^2}
  \\
  \nonumber 
  &= \frac{1}{1+\tilde h_u^2}\left( \frac{\tilde h_t^2}{2} \right)_{uu} 
    - \frac{\tilde h_{ut}^2}{1+\tilde h_u^2}
  + \frac1{L[\tilde h]}\bigg(
    \frac{1}{1+\tilde h_u^2}
    - 2\frac{\tilde h_u(1+\tilde h_{uu})}{(1+\tilde h_u^2)^2}
    \bigg)\left( \frac{\tilde h_t^2}{2} \right)_u
  \\&\quad\nonumber
    - \tilde h_t\frac{L'[\tilde h]}{L^2[\tilde h]}\frac{\tilde h_u}{1+\tilde h_u^2}
  \,.
  \end{align}
  Now, calculate
  \begin{align*}
  \partial_t \int_0^{\rho_0} \frac{\tilde h_t^2}{2}\,du
   &= \int_0^{\rho_0}\tilde h_t\,\tilde h_{tt}\,du
  \\
   &=
     \int_0^{\rho_0}
    \frac{1}{1+\tilde h_u^2}\left( \frac{\tilde h_t^2}{2} \right)_{uu} 
    - \frac{\tilde h_{ut}^2}{1+\tilde h_u^2}
  + \frac1{L[\tilde h]}\bigg(
    \frac{1}{1+\tilde h_u^2}
    - 2\frac{\tilde h_u(1+\tilde h_{uu})}{(1+\tilde h_u^2)^2}
    \bigg)\left( \frac{\tilde h_t^2}{2} \right)_u
    \,du
    \\&\qquad
    - \int_0^{\rho_0} \tilde h_t\frac{L'[\tilde h]}{L^2[\tilde h]}\frac{\tilde h_u}{1+\tilde h_u^2}
      \,du
  \\
   &=
    - \int_0^{\rho_0}
    \frac{\tilde h_{ut}^2}{1+\tilde h_u^2}
    \,du
    + 2\int_0^{\rho_0}
    \frac{\tilde h_u\,\tilde h_{uu}}{(1+\tilde h_u^2)^2}\left( \frac{\tilde h_t^2}{2} \right)_{u} 
    \,du
    + \bigg[
    \frac{1}{1+\tilde h_u^2}\left( \frac{\tilde h_t^2}{2} \right)_{u} 
      \bigg]_0^{\rho_0}
  \\&\qquad
    + \int_0^{\rho_0}
          \frac1{L[\tilde h]}\bigg(
    \frac{1}{1+\tilde h_u^2}
    - 2\frac{\tilde h_u(1+\tilde h_{uu})}{(1+\tilde h_u^2)^2}
    \bigg)\left( \frac{\tilde h_t^2}{2} \right)_u
    \,du
    \\&\qquad
    + (-L'[\tilde h]) \int_0^{\rho_0} \frac{1}{L^2[\tilde h]}\frac{\tilde h_u}{1+\tilde h_u^2}
      \bigg(\frac{\tilde h_{uu}}{1+\tilde h_u^2} + \frac1{L[\tilde h]}\frac{\tilde h_u}{1+\tilde h_u^2}\bigg)
      \,du
  \\
   &\le 
    - \frac12\int_0^{\rho_0}
    \frac{\tilde h_{ut}^2}{1+\tilde h_u^2}
    \,du
    + C\int_\gamma k^2\,ds
    + C(1 + \rho_0^{-2})
  \end{align*}
  where the constant $C>0$ depends only on $\vn{h_0}_{W^{2,\infty}}$.
  Note that in the above we used the boundary conditions $\tilde h_t({\rho_0},t) = 0$, $\tilde h_u(0,t) = 0$, and
  \begin{align*}
    (-L'[\tilde h]) &\int_0^{\rho_0} \frac{1}{L^2[\tilde h]}\frac{\tilde h_u}{1+\tilde h_u^2}
      \bigg(\frac{\tilde h_{uu}}{1+\tilde h_u^2} + \frac1{L[\tilde h]}\frac{\tilde h_u}{1+\tilde h_u^2}\bigg)
      \,du
  \\
    &= (-L'[\tilde h])\bigg\{ \bigg[-\frac12\frac{1}{L^2[\tilde h]}\frac{1}{1+\tilde h_u^2}\bigg]_0^{\rho_0}
     + \int_0^{\rho_0} \frac{1}{L^3[\tilde h]^2}\frac{\tilde h_u^2}{(1+\tilde h_u^2)^2}
      \,du
      \bigg\}
  \\&\le C\int_\gamma k^2\,ds
  \,.
  \end{align*}
  Recall that $\int_0^t\int_\gamma k^2\,ds\,dt' \le L(0)$ by Lemma \ref{lmlengthevo}.
  Then
  \begin{equation}
  \label{EQht}
  \int_0^{\rho_0} \frac{\tilde h_t^2}{2}\,du
  + \frac12\int_0^t \frac{\tilde h_{ut}^2}{1+\tilde h_u^2}\,dt'
  \le \int_0^{\rho_0} \frac{\tilde h_t^2}{2}\,du\bigg|_{t=0}
    + Ct + C
  \end{equation}
  where $C$ may denote a new constant that depends only on $\vn{h_0}_{W^{2,\infty}}$ (and $\rho_0$).

  Note that
  \begin{align*}
  \tilde h_t^2
  &= \frac{\tilde h_{uu}^2}{(1+\tilde h_u^2)^2}
    + 2L^{-1} \frac{\tilde h_u\tilde h_{uu}}{(1+\tilde h_u^2)^2}
    + L^{-2} \frac{\tilde h_u^2}{(1+\tilde h_u^2)^2}
  \\
  &= \left(\arctan(\tilde h_u)\right)_u^2
    + 2L^{-1} \frac{\tilde h_u\tilde h_{uu}}{(1+\tilde h_u^2)^2}
    + L^{-2} \frac{\tilde h_u^2}{(1+\tilde h_u^2)^2}
  \\
  &\ge
   \frac12\left(\arctan(\tilde h_u)\right)_u^2
    - L^{-2} \frac{\tilde h_u^2}{(1+\tilde h_u^2)^2}
  \,.
  \end{align*}
  Therefore the uniform gradient estimate and the above calculation yields the estimates
  \begin{align}\nonumber 
  &\int_0^{\rho_0} \left( \arctan(\tilde h_u)\right)_u^2\,du 
    \le C \qquad \forall t\in [0,T]\,,\text{ and}
  \\ \label{utx}
  &\int_0^T \int_0^{\rho_0} \left(\arctan(\tilde h_u)\right)_t^2\,du\,dt 
  =
  \int_0^T\int_0^{\rho_0} \frac{\tilde h_{ut}^2}{1+\tilde h_u^2} \,dx\,dt\le C\,, 
  \end{align}
  where now $C$ depends additionally on $T$.

  This means that the function $\arctan((h_n)_u)$ is uniformly bounded in 
  \[
    H^1([0,T],L^2([0,T]))\cap L^\infty([0,T],H^1([0,{\rho_0}]))
    \,.
  \]
  Therefore, the sequence $\arctan((h_n)_u)$ converges, up to a subsequence, to $\arctan(h_u)$ uniformly on $[0,{\rho_0}]\times [0,T]$. 
  Since $\arctan$ is injective this implies that the sequence $(h_n)_u$ converges to $h_u$ a.e. on $[0,{\rho_0}]\times [0,T]$,
  and we can pass to the limit in \eqref{eqdefw}, obtaining that $h$ is a weak solution of \eqref{eqgraph}.

  Since $\arctan({h_u})$ continuous, we have that $h_u$ is also continuous 
  (hence bounded) on $[0,{\rho_0}]\times [0,T]$.
  To finish the proof we wish to conclude a bound on $h_{uu}$, which would imply
  $h\in L^\infty([0,T],W^{2,\infty}([0,{\rho_0}]))$.

  To see this, we note that our above estimates (in particular the gradient estimate and \eqref{EQht}) imply
  \[
    \int k^2\,ds \le CT\,,
  \]
  for some universal constant $C$. Setting $\varphi = \frac{\tilde h_t^2}{2}$ and $\hat\varphi = \varphi e^{-t}$ we find
  \begin{align*}
  \bigg(
    \partial_t &- \frac{1}{1+\tilde h_u^2}\partial^2_{uu}
    - \frac1{L[\tilde h]}\bigg(\frac{1}{1+\tilde h_u^2} - 2\frac{\tilde h_u(1+\tilde h_{uu})}{(1+\tilde h_u^2)^2}\bigg)\partial_u
  \bigg)\hat\varphi
  \\&\le -\hat\varphi + C\sqrt{\hat\varphi}e^{-t/2}
  \\&\le -\frac12\hat\varphi + \frac{C^2}{2}
  \,.
  \end{align*}
  The boundary conditions imply $\hat\varphi_u(0) = 0$ and $\hat\varphi(\rho) = 0$.
  Using the Hopf Lemma then gives $\hat\varphi \le C^2$, which results in the estimate
  \[
    \tilde h_t^2 \le 2C^2e^t \le 2C^2e^T
    \,.
  \]
  In light of \eqref{eqgraph}, this implies a uniform bound on $h_{uu}$ and we are finished.
  \end{proof}

  Finally we note the following $W^{4,\infty}$ estimate, which is needed for the numerical analysis.

  \begin{prop}\label{propw52}
  Let $h_0\in W^{5,P}([0,{\rho_0}])$, $P\ge2$, with $(h_0)_u(0) = 0$ and $h_0({\rho_0}) = 0$.
  Then, the weak solution generated by Theorem \ref{thmw2inf} above satisfies
  \[
    \vn{h}_{W^{4,\infty}([0,{\rho_0}]))} \le C(h_0,\rho_0,T)
  \]
  and
  \[
    \vn{h}_{W^{5,P}([0,{\rho_0}]))} \le C(h_0,\rho_0,T,P)
    \,.
  \]
  \end{prop}
  \begin{proof}
  {\bf Step 1:} $W^{3,P}$ control.
  We begin by noting
  \begin{align*}
  h_{tt} &= \bigg(\frac{h_{uu}}{1+h_u^2} + \frac1{L[h]}\frac{h_u}{1+h_u^2}\bigg)_{t}
  \\
  &=
    \frac{h_{tuu}}{1+h_u^2}
    + \bigg(2 \frac{h_uh_{uu}}{(1+h_u^2)^2} + \frac1L \Big(\frac{1}{1+h_u^2} - 2\frac{h_u}{(1+h_u^2)^2}\Big)\bigg)h_{tu}
    - \frac{L'}{L^2} \frac{h_u}{1+h_u^2}
  \\
  \,.
  \end{align*}
  Using the uniform $W^{2,\infty}$ estimate from Theorem \ref{thmw2inf} we obtain
  \begin{equation}
  \label{EQh1}
  h_{tt}
   =
    \frac{1}{1+h_u^2}h_{tuu}
    + f(h)(1 + h_{tu})
  \end{equation}
  where $f$ is a bounded function, satisfying $|f|\le C$ on a bounded time
  interval $[0,T]$, with $C=C(T,h_0)$.

  Since our boundary conditions imply $h_{tt}$ vanishes at the Dirichlet boundary
  and $h_{tu}$ vanishes on the Neumann boundary, the product $h_{tu}h_{ttu}$
  vanishes on both boundaries. 
  Therefore
  \begin{align*}
  \frac{d}{dt} \int_0^{\rho_0} h_{tu}^{2N}\,du
   &= 2N\int_0^{\rho_0} h_{ttu} h_{tu}^{2N-1}\,du
  \\
   &= -2N(2N-1)\int_0^{\rho_0} h_{tt} h_{tuu} h_{tu}^{2N-2}\,du
  \\
   &=
        -2N(2N-1)\int_0^{\rho_0} h_{tuu}
      \bigg(
      \frac{1}{1+h_u^2}h_{tuu}
      + f(u)(1 + h_{tu})
    \bigg)
     h_{tu}^{2N-2}\,du
  \\
  &\le -2\delta N(2N-1) \int_0^{\rho_0} h_{tuu}^2 h_{tu}^{2N-2}\,du
    - 2N(2N-1)\int_0^{\rho_0} f(u)h_{tuu}\,h_{tu}^{2N-2}\,du
  \\&\qquad
    - 2N(2N-1)\int_0^{\rho_0} f(u)h_{tuu}\,h_{tu}^{2N-1}\,du
  \\
  &\le -\delta N(2N-1) \int_0^{\rho_0} h_{tuu}^2 h_{tu}^{2N-2}\,du
    + C_\delta N^2\int_0^{\rho_0} h_{tu}^{2N} + h_{tu}^{2N-2}\,du
  \,.
  \end{align*}
  On the last term we use the Young inequality to estimate
  \[
    \int_0^{\rho_0} h_{tu}^{2N-2}\,du
  \le C + C\int_0^{\rho_0} h_{tu}^{2N}\,du
  \,.
  \]
  Thus
  \begin{align*}
  \frac{d}{dt} \int_0^{\rho_0} h_{tu}^{2N}\,du
  &\le 
          C\int_0^{\rho_0} h_{tu}^{2N}\,du + C
  \,.
  \end{align*}
  This implies that for every $N\in\N$ we have
  \[
  \sup_{t\in[0,T]} \vn{h}_{W^{3,2N}}(t) \le C
  \]
  where $C$ depends only on $h_0$, $T$ and $N$.
  Note that we can not take $N\rightarrow\infty$; instead, we shall use this to next obtain a $W^{4,P}$ estimate.

  {\bf Step 2:} $W^{4,P}$ control.
  We begin by differentiating
  \[
  \frac{d}{dt} \int_0^{\rho_0} h_{tt}^{2N}\,du = 2N\int_0^{\rho_0} h_{t^3}h_{tt}^{2N-1}\,du
  \,.
  \]
  Now we wish to convert $h_{t^3}$ to $h_{ttuu}$ plus lower order terms:
  \begin{align}
    \label{EQh2}
  h_{t^3} &= \bigg(\frac{h_{uu}}{1+h_u^2} + \frac1{L[h]}\frac{h_u}{1+h_u^2}\bigg)_{tt}
  \\
  &=
    \bigg(\frac{h_{tuu}}{1+h_u^2}
    + \bigg(2 \frac{h_uh_{uu}}{(1+h_u^2)^2} + \frac1L \Big(\frac{1}{1+h_u^2} - 2\frac{h_u}{(1+h_u^2)^2}\Big)\bigg)h_{tu}
    - \frac{L'}{L^2} \frac{h_u}{1+h_u^2}
    \bigg)_t
  \\
  &=
    \frac{h_{ttuu}}{1+h_u^2}
    + f(u)\Big(
      h_{tu}h_{tuu}
      + h_{ttu}
      + h_{tu}^2
      + 1
    \Big)
  \,,
  \end{align}
  where again $f$ is a bounded function. Note that we can replace $h_{tuu}$ by $h_{tt}$ using \eqref{EQh1}.

  We have a number of additional terms to deal with.
  We continue with
  \begin{align*}
  \frac{d}{dt} \int_0^{\rho_0} h_{tt}^{2N}\,du
   &= 2N\int_0^{\rho_0} h_{tt}^{2N-1}
      \bigg(
    \frac{h_{ttuu}}{1+h_u^2}
    + f(u)\Big(
      h_{tu}h_{tt}
      + h_{ttu}
      + h_{tu}^2
      + 1
    \Big)
      \bigg)\,du
  \\
   &\le -\delta N(2N-1)\int_0^{\rho_0} h_{ttu}^2h_{tt}^{2N-2}\,du
    + 2NC\int_0^{\rho_0} 
    h_{tt}^{2N}|h_{tu}|
    \,du
  \\&\qquad
    + C_\delta N^2\int_0^{\rho_0} 
    h_{tt}^{2N}
    \,du
    + 2NC\int_0^{\rho_0} 
    h_{tt}^{2N}
    \,du
    + 2NC\int_0^{\rho_0} 
    h_{tu}^P
    \,du
    + C
  \\
   &\le -\delta N(2N-1)\int_0^{\rho_0} h_{ttu}^2h_{tt}^{2N-2}\,du
    + 2NC\int_0^{\rho_0} 
    h_{tt}^{2N}|h_{tu}|
    \,du
    + C\int_0^{\rho_0} 
    h_{tt}^{2N+1}
    \,du
    + C
  \,,
  \end{align*}
  where in the last step we used the $W^{3,P}$ estimates from above (valid for every $P$) and the Young inequality, for instance:
  \[
  h_{tt}^{2N}|h_{tu}| \le \varepsilon h_{tt}^{2N+1} + C_\varepsilon h_{tu}^{2N+1}
  \,.
  \]
  Note that we do not need to take $\varepsilon$ small here.

  Let us note the following interpolation procedure:
  \begin{align*}
  \int h_{tu}^Ph_{tt}^{Q}\,du
  &\le \int h_{tu}^Ph_{tuu}h_{tt}^{Q-1}(1+h_u^2)\,du
    + C\int h_{tu}^P (1+|h_{tu}|)h_{tt}^{Q-1}\,du
  \\
  &= \frac1{P+1}\int (h_{tu}^{P+1})_u h_{tt}^{Q-1}(1+h_u^2)\,du
    + C\int 
      h_{tu}^Ph_{tt}^{Q-2}
      + h_{tu}^{P+1}h_{tt}^{Q-2}
      \,du
  \\
  &=
    - \frac{Q-1}{P+1}\int h_{tu}^{P+1} h_{ttu} h_{tt}^{Q-2}(1+h_u^2)\,du
    - \frac2{P+1}\int h_{tu}^{P+1} h_{tt}^{Q-1}h_uh_{uu}\,du
  \\&\qquad
    + C\int 
      h_{tu}^Ph_{tt}^{Q-2}
      + h_{tu}^{P+1}h_{tt}^{Q-2}
      \,du
  \\
  &\le
      C\int h_{tu}^{P+1} |h_{ttu}|\,h_{tt}^{N-1+Q-N-1}\,du
  \\&\qquad
    + C\int 
      h_{tu}^{P+1}h_{tt}^{Q-1}
      + h_{tu}^Ph_{tt}^{Q-2}
      + h_{tu}^{P+1}h_{tt}^{Q-2}
      \,du
  \\
  &\le 
       \delta \int h_{ttu}^2h_{tt}^{2N-2}\,du
  \\&\qquad
   + C\int h_{tu}^{2P+2}h_{tt}^{2Q-2N-2}\,du
   + \frac12\int h_{tu}^Ph_{tt}^{Q}\,du
  \\&\qquad
    + C\int 
      (h_{tu}^{P}
      + h_{tu}^{P+1}
      + h_{tu}^{P+2})h_{tt}^{Q-2}
      \,du
  \,.
  \end{align*}
  There are no boundary terms in the above because $h_{tu}h_{tt}$ vanishes at both boundary points.
  Note also that we freely used the $W^{2,\infty}$ estimates.

  Absorbing, we find
  \[
  \int h_{tu}^Ph_{tt}^{Q}\,du
  \le 
       \delta \int h_{ttu}^2h_{tt}^{2N-2}\,du
   + C\int h_{tu}^{2P+2}h_{tt}^{2Q-2N-2}\,du
    + C\int 
      (h_{tu}^{P}
      + h_{tu}^{P+2})h_{tt}^{Q-2}
      \,du
  \]
  Applying this for $Q = 2N+1$, we see that the RHS involves products of powers of $h_{tu}$ and at most the $(2N-1)$-th power of $h_{tt}$.
  Using again Young's inequality and the $W^{3,P}$ estimates, we finally discover that
  \[
  \frac{d}{dt} \int_0^{\rho_0} h_{tt}^{2N}\,du
  \le C\int_0^{\rho_0} h_{tt}^{2N}\,du + C
  \]
  which
  implies that for every $N\in\N$ we have
  \[
  \sup_{t\in[0,T]} \vn{h}_{W^{4,2N}}(t) \le C
  \]
  where $C$ depends only on $h_0$, $T$ and $N$.

  {\bf Step 3:} $W^{5,P}$ control.
  We begin by noting that the $W^{4,P}$ estimates yield a \emph{uniform} $W^{3,\infty}$ estimate, which means that now
  \begin{align*}
  h_{ttt} &= \bigg(\frac{h_{uu}}{1+h_u^2} + \frac1{L[h]}\frac{h_u}{1+h_u^2}\bigg)_{tt}
  \\
  &=
    \bigg(
    \frac{h_{tuu}}{1+h_u^2}
    + \bigg(2 \frac{h_uh_{uu}}{(1+h_u^2)^2} + \frac1L \Big(\frac{1}{1+h_u^2} - 2\frac{h_u}{(1+h_u^2)^2}\Big)\bigg)h_{tu}
    - \frac{L'}{L^2} \frac{h_u}{1+h_u^2}
    \bigg)_t
  \\
  &= \frac{h_{ttuu}}{1+h_u^2} + f(h)(1 + h_{ttu} + h_{tuu})
  \,,
  \end{align*}
  where $f$ is a bounded function.

  Using this we calculate
  \begin{align*}
  \frac{d}{dt} \int_0^{\rho_0} h_{ttu}^{2N}\,du
  &= -2N(2N-1) \int_0^{\rho_0} h_{ttt} h_{ttuu} h_{ttu}^{2N-2}\,du
  \\
  &= -2N(2N-1) \int_0^{\rho_0} h_{ttuu} h_{ttu}^{2N-2}
        \bigg(
        \frac{h_{ttuu}}{1+h_u^2} + f(h)(1 + h_{ttu} + h_{tuu})
        \bigg)\,du
  \\
  &\le -2N(2N-1)\delta \int_0^{\rho_0} h_{ttuu}^2 h_{ttu}^{2N-2} \,du
  \\&\qquad
     + 2CN(2N-1) \int_0^{\rho_0} h_{ttuu} h_{ttu}^{2N-2}
        (1 + |h_{ttu}| + |h_{tuu}|)
        \,du
  \\
  &\le -N(2N-1)\delta \int_0^{\rho_0} h_{ttuu}^2 h_{ttu}^{2N-2} \,du
  \\&\qquad
     + 2CN(2N-1) \int_0^{\rho_0} 
        (1 + h_{ttu}^2 + h_{tuu}^2)
        h_{ttu}^{2N-2}
        \,du
  \,.
  \end{align*}
  Employing again the Young inequality we estimate
  \[
       2CN(2N-1) \int_0^{\rho_0} 
        (1 + h_{ttu}^2 + h_{tuu}^2)
        h_{ttu}^{2N-2}
        \,du
  \le
    \int_0^{\rho_0} h_{ttu}^{2N}\,du
    + \int_0^{\rho_0} h_{tuu}^{N}\,du
    + C
  \le
    \int_0^{\rho_0} h_{ttu}^{2N}\,du
    + C
  \]
  which finally yields
  \begin{align*}
  \frac{d}{dt} \int_0^{\rho_0} h_{ttu}^{2N}\,du
  \le C\int_0^{\rho_0} h_{ttu}^{2N}\,du + C
  \end{align*}
  and the desired $W^{5,P}$ estimate for any $P$.
  This in turn yields uniform $W^{4,\infty}$ control, and the proof of the proposition is ended.
  \end{proof}

  The technique used in the previous proof for estimates in $W^{4,P}$ and
  $W^{5,P}$ can be iterated, yielding the following analogous higher-order
  statement.

  \begin{prop}\label{propwkp}
  Let $h_0\in W^{K,P}([0,{\rho_0}])$, $K\ge3$, $P\ge2$, with $(h_0)_u(0) = 0$ and $h_0({\rho_0}) = 0$.
  Then, the weak solution generated by Theorem \ref{thmw2inf} above satisfies
  \[
    \vn{h}_{W^{K-1,\infty}([0,{\rho_0}]))} \le C(h_0,\rho_0,T)
  \]
  and
  \[
    \vn{h}_{W^{K,P}([0,{\rho_0}]))} \le C(h_0,\rho_0,T,P)
    \,.
  \]
  \end{prop}
  \begin{proof}
  For the proof, follow the proof of Theorem \ref{thmallest} of Case 1 up to
  \eqref{EQcase1} then integrate from zero to $t$; and for Case 2 similarly up to
  \eqref{EQcase2} then integrate from zero to $t$.
  \end{proof}

\section{Convergence and smoothing}

In order to conclude the convergence, we use the exponential decay of area.
The idea is to combine this exponential decay with controlled growth of other
derivatives and a standard interpolation inequality.

For this control on higher derivatives, the key point (in contrast to the
situation in Propositions \ref{propw52} and \ref{propwkp}) is that we only require
these to hold for \emph{large time}.
This means that we can get away with requiring much weaker regularity on our
initial data for these estimates.
This is a manifestation of the \emph{smoothing effect} for the flow, outlined
in the following theorem.

\begin{thm}
\label{thmallest}
Let $h_0\in W^{2,\infty}([0,{\rho_0}])$ with $(h_0)_u(0) = 0$ and $h_0({\rho_0}) = 0$.
Then \eqref{eqgraph} admits a smooth classical solution $h:[0,{\rho_0}]\times(0,\infty)\rightarrow\R$.
Furthermore, all derivatives of the graph function obey the following estimate:
\[
\vn{\partial_t^p \partial_u^q h}_2^2(t) \le C(p,q,h_0)t\,,\qquad\text{ for $t>1$}\,.
\]
\end{thm}
\begin{proof}
First, Theorem \ref{thmw2inf} yields the existence of a $W^{2,\infty}$ solution.
In this proof we must control all higher derivatives of $h$.

We shall use induction.
When considering estimates (and existence of) derivatives of order $N+1$, we will be finished if we can deal with the following two cases:
\[
\text{{\bf Case 1}: $N$ is even}\,,\quad\text{and}\quad
\text{{\bf Case 2}: $N$ is odd}\,.
\]
Note that we may always trade one time derivative for two space derivatives plus a lower order term (using the PDE \eqref{eqgraph}).
This means that we will be finished if we can obtain existence and control of:
\begin{itemize}
\item[{\bf (Case 1)}] $h_{t^pu}$ for $p = N/2$;
\item[{\bf (Case 2)}] $h_{t^p}$ for $p = (N+1)/2$.
\end{itemize}
Each case is slightly different, with the general strategy following the proof of Proposition \ref{propw52}. We will endeavour to build on this proof and not spend too much time repeating ideas.

Note that our previous work already takes us to estimates of order 5 (so $N=5$).
We proceed from this case.

{\bf Case 1.} $N$ is even.

Set $P=N/2$. We begin by noting that the $W^{N,P}$ estimates yield a \emph{uniform} $W^{N-1,\infty}$ estimate, which means that we have uniform estimates on
\[
	h_{t^ku^l}
\]
for all $k$, $l$ such that $2k + l \le N-1$.
Using this, we find
\begin{align}
h_{t^{P+1}} &= \bigg(\frac{h_{uu}}{1+h_u^2} + \frac1{L[h]}\frac{h_u}{1+h_u^2}\bigg)_{t^{P}}
\notag
\\
&=
	\bigg(
	\frac{h_{tuu}}{1+h_u^2}
	+ \bigg(2 \frac{h_uh_{uu}}{(1+h_u^2)^2} + \frac1L \Big(\frac{1}{1+h_u^2} - 2\frac{h_u}{(1+h_u^2)^2}\Big)\bigg)h_{tu}
	- \frac{L'}{L^2} \frac{h_u}{1+h_u^2}
	\bigg)_{t^{P-1}}
\notag
\\
&= \frac{h_{t^Puu}}{1+h_u^2} + f(h)(1 + h_{t^Pu} + h_{t^P})
\label{EQhtp}
\,,
\end{align}
where $f$ is a bounded function.
Therefore
\begin{align*}
\frac{d}{dt} \int_0^{\rho_0} h_{t^{P}u}^{2Q}\,du
&= -2Q(2Q-1) \int_0^{\rho_0} h_{t^{P+1}} h_{t^Puu} h_{t^Pu}^{2Q-2}\,du
\\
&= -2Q(2Q-1) \int_0^{\rho_0} h_{t^Puu} h_{t^Pu}^{2Q-2}
			\bigg(
			\frac{h_{t^Puu}}{1+h_u^2} + f(h)(1 + h_{t^Pu} + h_{t^P})
			\bigg)\,du
\\
&\le -2Q(2Q-1)\delta \int_0^{\rho_0} h_{t^Puu}^2 h_{t^Pu}^{2Q-2} \,du
\\&\qquad
   + 2CQ(2Q-1) \int_0^{\rho_0} h_{t^Puu} h_{t^Pu}^{2Q-2}
			(1 + |h_{t^Pu}| + |h_{t^P}|)
			\,du
\\
&\le -Q(2Q-1)\delta \int_0^{\rho_0} h_{t^Puu}^2 h_{t^Pu}^{2Q-2} \,du
\\&\qquad
   + 2CQ(2Q-1) \int_0^{\rho_0} 
			(1 + h_{t^Pu}^2 + h_{t^P}^2)
 			h_{t^Pu}^{2Q-2}
			\,du
\,.
\end{align*}
Young's inequality and $W^{N,Q}$-estimates imply
\[
     \int_0^{\rho_0} 
			(1 + h_{t^Pu}^2 + h_{t^P}^2)
 			h_{t^Pu}^{2Q-2}
			\,du
\le
	C\int_0^{\rho_0} h_{t^Pu}^{2Q}\,du
	+ C\int_0^{\rho_0} h_{t^P}^{Q}\,du
	+ C
\le
	C\int_0^{\rho_0} h_{t^Pu}^{2Q}\,du
	+ C
\]
which finally yields
\begin{align}
\label{EQcase1}
\frac{d}{dt} \int_0^{\rho_0} h_{t^Pu}^{2Q}\,du
	+ \int_0^{\rho_0} h_{t^Puu}^2 h_{t^Pu}^{2Q-2} \,du
\le C\int_0^{\rho_0} h_{t^Pu}^{2Q}\,du + C
\,.
\end{align}
Integrating this from zero to $t$ yields the estimate needed for Proposition \ref{propwkp}; here, we instead interpolate further on the RHS:
\begin{align*}
\int_0^{\rho_0} h^{2Q}_{t^Pu}\,du
	&= (2Q-1)\int_0^{\rho_0} h_{t^Puu}h_{t^P}h_{t^Pu}^{2Q-2}\,du
\\
	&\le \frac{\varepsilon}{2}\int_0^{\rho_0} h_{t^Puu}^2h_{t^Pu}^{2Q-2}\,du
	 + \frac{(2Q-1)^2}{2\varepsilon}\int_0^{\rho_0} h_{t^P}^2h_{t^Pu}^{2Q-2}\,du
\\
	&\le \frac{\varepsilon}{2}\int_0^{\rho_0} h_{t^Puu}^2h_{t^Pu}^{2Q-2}\,du
	 + \frac12\int_0^{\rho_0} h_{t^Pu}^{2Q}\,du
 	 + \frac{C_Q}{\varepsilon^Q}\int_0^{\rho_0} h_{t^P}^{2Q}\,du
\,.
\end{align*}
Note that in the last step we used Young's inequality.
Absorbing yields
\begin{equation}
\label{EQinterp}
\int_0^{\rho_0} h^{2Q}_{t^Pu}\,du
	\le \varepsilon\int_0^{\rho_0} h_{t^Puu}^2h_{t^Pu}^{2Q-2}\,du
 	 + \frac{C_Q}{\varepsilon^Q}\int_0^{\rho_0} h_{t^P}^{2Q}\,du
\,.
\end{equation}
Using \eqref{EQinterp} in \eqref{EQcase1} then absorbing again yields
\[
\frac{d}{dt} \int_0^{\rho_0} h_{t^Pu}^{2Q}\,du
	+ \frac12\int_0^{\rho_0} h_{t^Puu}^2 h_{t^Pu}^{2Q-2} \,du
\le C
\,.
\]
Integrating yields
\begin{equation}
	\label{theest}
\int_0^{\rho_0} h_{t^Pu}^{2Q}\,du
	+ \frac12\int_{\varepsilon_N}^{\hat t}\int_0^{\rho_0} h_{t^Puu}^2 h_{t^Pu}^{2Q-2} \,du\,dt
\le C\hat t + \int_0^{\rho_0} h_{t^Pu}^{2Q}\,du\bigg|_{t=\varepsilon_N}
\,.
\end{equation}
Note that the inductive hypothesis implies that for all $\varepsilon_N>0$, $\int_0^{\rho_0} h_{t^Pu}^{2}\,du\bigg|_{t=\varepsilon_N}$ exists and is bounded (this is the $Q=1$ case of the \emph{previous step}).

Using estimate \eqref{theest} yields that the integral
\[
\int_0^{\rho_0} h_{t^Puu}^2 \,du
\]
instantaneously exists and is bounded. This implies that
\[
\int_0^{\rho_0} h_{t^{P}u}^{2Q} \,du\,,\quad
\text{ and }
\int_0^{\rho_0} h_{t^{P+1}}^2 \,du
\]
exist and are bounded for each $t>0$ by a constant that (possibly) blows up as $t\searrow0$.
Therefore $\int_0^{\rho_0} h_{t^Pu}^{2Q}\,du\bigg|_{t=\varepsilon_N}$ for each $Q$ and $\varepsilon_N>0$ exists and is bounded.
This yields immediately the linear in time estimate claimed for estimates of order $2Q=N$,
and the requirements of the proceeding inductive step.

{\bf Case 2.} $N$ is odd.

Set $P=(N+1)/2$. As before, we begin by noting that the $W^{N,P}$ estimates yield a \emph{uniform} $W^{N-1,\infty}$ estimate, which means that we have uniform estimates on
\[
	h_{t^ku^l}
\]
for all $k$, $l$ such that $2k + l \le N-1$.
We shall use a version of \eqref{EQhtp} as before, but this time our induction hypothesis is that we have estimates of one order less. This means that we need to consider an extra term:
%
%
\begin{align*}
h_{t^{P+1}}
&=
	\bigg(
	\frac{h_{tuu}}{1+h_u^2}
	+ \bigg(2 \frac{h_uh_{uu}}{(1+h_u^2)^2} + \frac1L \Big(\frac{1}{1+h_u^2} - 2\frac{h_u}{(1+h_u^2)^2}\Big)\bigg)h_{tu}
	- \frac{L'}{L^2} \frac{h_u}{1+h_u^2}
	\bigg)_{t^{P-1}}
\\
&= \frac{h_{t^Puu}}{1+h_u^2} + f(h)(1 + h_{t^Pu} + h_{t^P} + h_{t^{P-1}u})
\,,
\end{align*}
where as before $f$ is a bounded function.

Calculating,
\begin{align*}
\frac{d}{dt} \int_0^{\rho_0} h_{t^{P}}^{2Q}\,du
&= 2Q\int_0^{\rho_0} h_{t^{P+1}} h_{t^P}^{2Q-1}\,du
\\
&= 2Q\int_0^{\rho_0} \frac{h_{t^Puu}}{1+h_u^2}h_{t^P}^{2Q-1}\,du
\\&\qquad
 + 2Q\int_0^{\rho_0} f(h)(1 + h_{t^Pu} + h_{t^P} + h_{t^{P-1}u})h_{t^P}^{2Q-1}\,du
\\
&= -2Q(2Q-1)\int_0^{\rho_0} \frac{h_{t^Pu}^2}{1+h_u^2}h_{t^P}^{2Q-2}\,du
\\&\qquad
 + 2Q\int_0^{\rho_0} \hat f(h)(1 + h_{t^Pu} + h_{t^P} + h_{t^{P-1}u})h_{t^P}^{2Q-1}\,du
\\
&\le -2Q(2Q-1)\delta \int_0^{\rho_0} h_{t^Pu}^2 h_{t^P}^{2Q-2} \,du
\\&\qquad
   + 2CQ(2Q-1) \int_0^{\rho_0} |h_{t^Pu}|\, |h_{t^P}|^{2Q-1}
                             + h_{t^P}^{2Q}
                             + |h_{t^P}|^{2Q-1} + |h_{t^{P-1}u}|\,|h_{t^P}|^{2Q-1}
			\,du
\\
&\le -Q(2Q-1)\delta \int_0^{\rho_0} h_{t^Pu}^2 h_{t^P}^{2Q-2} \,du
   + C\int_0^{\rho_0} h_{t^P}^{2Q}
			\,du
	+ C
\,.
\end{align*}
That is,
\begin{align}
\label{EQcase2}
\frac{d}{dt} \int_0^{\rho_0} h_{t^P}^{2Q}\,du
	+ \int_0^{\rho_0} h_{t^Pu}^2 h_{t^P}^{2Q-2} \,du
\le C\int_0^{\rho_0} h_{t^P}^{2Q}\,du + C
\,.
\end{align}
Integrating this from zero to $t$ yields the estimate needed for Proposition \ref{propwkp}.
As in Case 1, we need to instead interpolate further on the RHS; this time, we
need to first introduce some space derivatives, and the estimate becomes more
complicated.
This is done by again another version of \eqref{EQhtp}; recall the uniform $W^{N-1,\infty}$ estimate allows us to put any term with a highest order derivative of $(2P-2)$ or less in the function $f$:
\[
h_{t^P}
 = \frac{h_{t^{P-1}uu}}{1+h_u^2}
	+ f(h)(1 + h_{t^{P-1}u})
\]
where again $f$ is a bounded function.
Thus
\begin{align*}
\int_0^{\rho_0} h_{t^P}^{2Q}\,du
&= \int_0^{\rho_0} \frac{h_{t^{P-1}uu}}{1+h_u^2} h_{t^P}^{2Q-1}\,du 
	+ \int_0^{\rho_0} f(h)(1 + h_{t^{P-1}u}) h_{t^P}^{2Q-1}\,du
\\
&\le -\delta\int_0^{\rho_0} h_{t^{P-1}u}^2h_{t^P}^{2Q-2}\,du 
	+ \int_0^{\rho_0} \hat f(h)(1 + h_{t^{P-1}u}) h_{t^P}^{2Q-1}\,du
\\
&\le -\delta\int_0^{\rho_0} h_{t^{P-1}u}^2h_{t^P}^{2Q-2}\,du 
	+ \frac12 \int_0^{\rho_0} h_{t^P}^{2Q}\,du
	+ C
	+ \int_0^{\rho_0} |h_{t^{P-1}u}| |h_{t^P}|^{2Q-1}\,du
\\
&\le -\delta\int_0^{\rho_0} h_{t^{P-1}u}^2h_{t^P}^{2Q-2}\,du 
	+ \frac34 \int_0^{\rho_0} h_{t^P}^{2Q}\,du
	+ C
\end{align*}
Note that we used the $W^{2P-1,Q}$ estimates and Young's inequality.
Absorbing, we find
\[
\int_0^{\rho_0} h_{t^P}^{2Q}\,du
\le C\,.
\]
Returning to our evolution equation, this implies
\[
\frac{d}{dt} \int_0^{\rho_0} h_{t^P}^{2Q}\,du
	+ \frac12\int_0^{\rho_0} h_{t^Pu}^2 h_{t^Pu}^{2Q-2} \,du
\le C
\,.
\]
Integrating yields
\begin{equation}
	\label{theest2}
\int_0^{\rho_0} h_{t^P}^{2Q}\,du
	+ \int_{\varepsilon_N}^{\hat t}\int_0^{\rho_0} h_{t^Pu}^2 h_{t^P}^{2Q-2} \,du\,dt
\le C\hat t + \int_0^{\rho_0} h_{t^P}^{2Q}\,du\bigg|_{t=\varepsilon_N}
\,.
\end{equation}
Note that for all $\varepsilon_N>0$, $\int_0^{\rho_0}
h_{t^P}^{2}\,du\bigg|_{t=\varepsilon_N}$ exists and is
bounded (this is the $Q=1$ case of the \emph{previous step}).

As before, we apply the estimate first for $Q=1$, to find that
\[
\int_0^{\rho_0} h_{t^Pu}^2 \,du
\]
instantaneously exists and is bounded. This will be used in the next step; for now, it implies that
\[
\int_0^{\rho_0} h_{t^{P}}^{2Q} \,du
\]
for any $Q$ exist and are bounded for each $t>0$ by a constant that (possibly)
blows up as $t\searrow0$.
Therefore $\int_0^{\rho_0} h_{t^P}^{2Q}\,du\bigg|_{t=\varepsilon_N}$ for each
$Q$ and $\varepsilon_N>0$ exists and is bounded.
This combined with \eqref{theest2} yields immediately the linear in time
estimate claimed for estimates of order $2Q=N$, and the requirements of the
proceeding inductive step.

Iterating through the cases with induction we find that every derivative of $h$
exists instantaneously.
Furthermore, as noted in the argument above we have the bound
\[
\vn{\partial_t^p \partial_u^q h}_2^2(t) \le Ct
\]
for all $t>\varepsilon_{2p+q}$.
By taking $\varepsilon_{2p+q} = 1$, we are able to guarantee that the above estimate holds for $t>1$, which gives the estimate claimed in the theorem.
\end{proof}

\begin{rmk}
Naturally, there is nothing special about $t=1$ in the smoothing estimate.
We could choose any fixed time.
We may obtain an estimate for all $t>0$, however the constant would not be
uniformly bounded for $t$ small (even if the initial graph function is
regular).
\end{rmk}

We need a custom-built interpolation inequality to move between high-order derivatives of $h$ and the enclosed area.

\begin{prop}
\label{ourinter}
Let $\gamma:[0,L]\rightarrow\R^2$ be a smooth regular curve with length $L$ satisfying the graphicality condition
\[
	G_{-e_2}(s) \ge C_G\,,
\]
and the positivity condition
\[
	\IP{\gamma(s)}{e_2} \ge 0\,.
\]
Then, for the associated graph function $h:[a,b]\rightarrow\R$, we have
\[
	\vn{h_{x^j}}_{L^2(dx)} \le C(A(\gamma))^\frac{1}{2}\vn{h_{x^{2j}}}_{L^2(dx)}^\frac{1}{2} + CA(\gamma)
	\,,
\]
where $j\in\N$, $C$ is a constant depending only on $a, b$, and $A(\gamma)$
denotes the area bounded by the curve $\gamma$ and the $e_1$-axis.
\end{prop}
\begin{proof}
Graphicality implies that there exist $a, b\in\R$ and a graph function $h:[a,b]\rightarrow\R$ such that
\[
	\gamma(s) = (x,h(x))
	\,.
\]
The Gagliardo-Nirenberg Sobolev inequality \cite{GNS1,GNS2,GNS3} applied to $h$ yields
\[
	\bigg(\int_a^b |h_{x^j}(x)|^2\,dx\bigg)^\frac12
	\le	C\bigg(\int_a^b |h(x)|\,dx\bigg)^\frac{1}{2} \bigg(\int_a^b |h_{x^{2j}}(x)|^2\,dx\bigg)^\frac{1}{4}
	      + C\int_a^b |h(x)|\,dx
	      \,, 
\]
for a constant $C$ depending only on $a$ and $b$.

The positivity condition implies $h(x) \ge 0$ so $|h(x)| = h(x)$.
Using $A(\gamma)$ to denote the area of the region bounded by $\gamma([0,L])$ and the $e_1$-axis, we have
\[
	A(\gamma) = \int_a^b |h(u)|\,du
\,.
\]
Therefore we find
\[
	\vn{h_{x^j}}_{L^2(dx)} \le C(A(\gamma))^\frac{1}{2}\vn{h_{x^{2j}}}_{L^2(dx)}^\frac{1}{2} + CA(\gamma)
	\,,
\]
as required.
\end{proof}

We finish with the convergence result, which is by now a corollary of our earlier work.

\begin{proof}[Proof of Theorem \ref{THMmain}.]
First, we use the above results to generate a solution from the $W^{2,\infty}$ data.
We have $T=\infty$ by Theorem \ref{LTE}.
Our interpolation inequality (Proposition \ref{ourinter}), the exponential
decay of area (Proposition \ref{DC_area evolution}), and the linear-in-time
estimates (Theorem \ref{thmallest}) give
\[
	\vn{\partial_u^jh}_{2}(t) \le Ct^\frac14e^{-\frac{C_G^2}{2\rho_0L(0)}t} + Ce^{-\frac{C_G^2}{\rho_0L(0)}t}
	\,,
\]
Therefore all spatial derivatives of $h$ converge exponentially fast to zero.
The evolution equation for $h$ then yields exponential decay of all temporal derivatives of $h$ also.
This is the desired exponentially fast convergence in the smooth topology of
the solution to the horizontal line segment connecting the origin to
$(\rho_0,0)$.
\end{proof}

\section{Numerical approximation of solutions}
  \label{sec_numerical_algorithm}

We now describe a numerical scheme to compute an approximation to equation \eqref{eqgraph} and prove stability, consistency and convergence results.

Let $m,n\in \N$ and let $\delta t= T/ m$, $\delta u =\rho_0 / n$. 
Let 
\[
S=\left\{i\delta u: 0\leq i\leq n \right\},\quad T=\left\{j\delta t: 0\leq j< m \right\},
\]
and $M=S\times \left(T\cup\{m\}\right)$.

We call the domain of any discrete function a mesh.
Let $F(M,\R)$ be the set of all functions $w:M\to\R$.
We shall write $w(k\delta u, j\delta t)$ as $w^j_k$ and $(k,j)\in M$ rather than $(k\delta u, j\delta t)\in M$.
If $w\in F(M,\R)$ we write $w_\cdot^j$ for the
mesh function with domain $S$ or some appropriate restriction of $S$.
On occasion, where $j$ is understood, we write $w_k$ rather than $w_k^j$, as this
reduces substantial notational clutter.
Let $\pi:[0,\rho_0]\times [0,T)\to M$ be the natural restriction.
If $f:[0,\rho_0]\times [0,T)\to \R$ we write $\pi(f)$ for the mesh function $f|_M$.

Let $D_0, D_+, D_-$ be the standard spatial derivative operators on $F(M,\R)$ given, for example, in \cite[Equation 1.1.2]{gustafsson2013time}.
Let $D^2=D_+D_-$ and let $D_t$ be the first order forward approximation 
to the time derivative. We remind the
reader that
$D_+$ is not defined at $n$,
$D_-$ is not defined at $0$, $D^2$ is not defined on either boundary,
and $D_t$ is not defined at $m$.
When taking the maximum norm we implicitly restrict domains where needed.

  Our finite difference approximation to equation \eqref{eqgraph}
  is a mesh function $w\in{F}(M,\R)$
  such that, for all $k\in S$ and $j\in T$,
  \begin{align}
      w^0_k&=h(k\delta u, 0), \label{eq_full_scheme_2_initial}\\
      w^j_n&=0, \label{eq_full_scheme_2_boundary_n}
  \end{align}
  for all $j\in T$,
  \begin{equation}
      \label{eq_num_boundary_condition}
      w^{j+1}_0 = w^j_0 + \delta t\, (D_t w)_1^{j},
  \end{equation}
  and
  for all $k\in S\setminus\{0,n\}, j\in T$,
  \begin{align}
    (D_t w)_k^{j} &= 
        \frac{1}{1 + \left(\left(D_0w\right)_k^j\right)^2}
        \left(
          \left(D^2w\right)_k^j 
          + 
          \frac{\left(D_0w\right)_k^j}{L[w^j_{\cdot}]}
        \right),
      \label{eq_full_scheme_2}
  \end{align}
  where the operator $L:F(S,\R)\to\R$ is the discrete length functional
  \[
    L[v_\cdot]=\delta u\sum_{i=0}^{n-1}\sqrt{%
      1 + \left((D_+v)_{i}\right)^2%
    }\ \,.
  \]
  We define
  ${}^w\!A_{k}^j = \frac{1}{1 + \left((D_0)w_k^j\right)^2}$
  and, when context allows, drop the superscript $w$.

  The difference of
  non-local terms in our numeric scheme is controlled by
  the norm of the $D_+$ derivative.
  \begin{lem}
    \label{lem_L_dif_numer}
    If $w,v\in F(M,\R)$ then
    $\abs{{L(w^{j}_\cdot)}-{L(v^{j}_\cdot)}}
      \leq 
        \rho_0\norm{(D_+(w - v))_{\cdot}^j}_{L^\infty}$.
  \end{lem}
  \begin{proof}
    This follows by direct computation
    once we note that $\sqrt{1+x^2}-\sqrt{1+y^2}<\abs{x-y}$.
  \end{proof}

  The following is a consistency result for
  our numeric length approximation.
  \begin{lem}
    \label{lem_length_diff_bound}
    If $h$ is a solution to equation \eqref{eqgraph}
    in $C^2([0,\rho_0]\times[0,T))$
    and there exists $B\in[0,\infty)$
    so that, for all $(x,j)\in [0,\rho_0]\times T$,
    $\max\{\abs{h_{uu}(x, j\delta t)}\}\leq B$
    then, for all $j\in T$,
    \[
        \frac{1}{L[h(\cdot, j\delta t)]}
        -
        \frac{1}{L[(\pi(h))_\cdot^j]}
        \leq \frac{3}{2\rho_0}B\delta u.
    \]
  \end{lem}
  \begin{proof}
    By definition $L[h(\cdot, j\delta t)]\leq \rho_0$ and
    $L[(\pi(h))_\cdot^j]\leq\rho_0$. 
    The result now follows 
    from the mean value theorem,
    an application of
    Taylor's theorem with remainder 
    and as
    $\sqrt{1+x^2}-\sqrt{1+y^2}<\abs{x-y}$.
  \end{proof}

To prove convergence we make use of the discrete maximum principle (see \cite{ciarlet1970discrete, farago2012discrete} and in particular \cite[Theorem 9]{mincsovics2010discrete}).

\begin{lem}\label{lem_max_inf_numerical_gen}
  Let $X,Y\in F(M,\R)$ be arbitrary mesh functions.
  If $w\in F(M,\R)$ is such that, for all $k\in S\setminus\{0,n\}$ and $j\in T$,
  \begin{align*}
        (D_t w)_k^{j} &\leq
          X_k^j \left(D^2w\right)_k^j 
              + Y_k^j{\left(D_0w\right)_k^j}
  \end{align*}
  and, for all $k\in S\setminus\{0,n\}$ and $j\in T$,
  \begin{align*}
        \left(X^j_k + \frac{Y_{k}^j}{2}\delta u\right)\geq 0, \quad
        \left(1 - 2 X_k^j \frac{\delta t}{\delta u^2}\right) \geq 0,\quad
        \left(X^j_k - \frac{Y_{k}^j}{2}\delta u\right)\geq 0,
  \end{align*}
  then the vector maximum norm of $w$ is such that, for all $j\in T$,
  \[
        \norm{w^{j+1}_\cdot}_{L^\infty}
        \leq
        \max\left\{
          \abs{w_{0}^{j}}_{L^\infty},
          \norm{w_{\cdot}^{j}}_{L^\infty},
          \abs{w_{n}^{j}}_{L^\infty}
        \right\}.
  \]
\end{lem}
\begin{proof}
    We compute that, for all $k\in S\setminus\{0,n\}$
    and $j\in T\setminus\{m\}$,
    \begin{align*}
      w_k^{j+1} &\leq w_k^j + \delta t
      \left(
        X_k^j \left(D^2w\right)_k^j + Y_k^j{\left(D_0w\right)_k^j}
      \right)\\
      &= w_k^j + 
        X_k^j\frac{w_{k+1}^j - 2w_{k}^{j} + w_{k-1}^j}{(\delta u)^2}\delta t
        + Y_{k}^j\frac{w_{k+1}^j - w_{k-1}^j}{2\delta u}\delta t\\
      &= \left(X^j_k + \frac{Y_{k}^j}{2}\delta u\right)
        \frac{\delta t}{(\delta u)^2} w_{k+1}^j+
      \left(1 - 2 X_k^j \frac{\delta t}{\delta u^2}\right) w_k^j + 
      \left(X^j_k - \frac{Y_{k}^j}{2}\delta u\right)
        \frac{\delta t}{(\delta u)^2} w_{k-1}^j.
    \end{align*}
    Hence, with our assumptions,
    \begin{align*}
      \abs{w_k^{j+1}} 
      &= 
      \left(
        1 
        - 2 X_k^j \frac{\delta t}{(\delta u)^2}
        + X^j_k\frac{\delta t}{(\delta u)^2} 
        + X^j_k\frac{\delta t}{(\delta u)^2} 
        + \frac{Y_{k}^j}{2}\frac{\delta t}{\delta u}
        - \frac{Y_{k}^j}{2}\frac{\delta t}{\delta u}
      \right)\norm{w_{\cdot}^{j}}_{L^\infty}\\
      &\leq \norm{w_{\cdot}^{j}}_{L^\infty}.
    \end{align*}
    Therefore
    \[
      \norm{w^{j+1}_\cdot}_{L^\infty}
      \leq\max\left\{
        \abs{w_0^{j+1}}, \abs{w_n}^{j+1}, \norm{w_{\cdot}^{j}}_{L^\infty}
      \right\}.
    \]
  \end{proof}

  Using the discrete maximum principle we can show that
  bounded initial and boundary data implies existence of
  a solution to our numeric scheme.
  \begin{thm}
    \label{thm_inf_bound_w_numeric}
    If $w\in F(M,\R)$ satisfies
    equation \eqref{eq_full_scheme_2}
    and $\delta u,\delta t$ are such that
    $(\delta u)^2 \geq 2\delta t$,
    and
    $2\rho_0\geq \delta u$
    then,
    for all $j\in T$,
    \[
      \norm{w^{j+1}_\cdot}_{L^\infty}
      \leq
      \max\left\{
        \norm{w_{\cdot}^{j}}_{L^\infty},
        \abs{w_{0}^{j+1}},
        \abs{w_{n}^{j+1}}
      \right\}.
    \]
  \end{thm}
  \begin{proof}
    Lemma
    \ref{lem_max_inf_numerical_gen}
    and equation \eqref{eq_full_scheme_2}
    give the result.
  \end{proof}

  The following records a useful computation.
  \begin{prop}
    \label{prop_DtD+_is}
    If $w\in F(M,\R)$ satisfies
    equations 
    \eqref{eq_full_scheme_2} then
    for all $j\in T$
    and all $k\in S\setminus\{0,n\}$
    then
    \begin{align*}
        (D_tD_+w)^{j+1}_k
      &=
        X^j_k\,(D^2D_+w)^j_k + Y^j_k\,(D_0D_+w)^j_{k}
    \end{align*}
    where
    \begin{align*}
        {}^w\!X^j_k 
      &= 
        \frac{1}{2}\left(A_{k}^j + A^j_{k+1}\right),\\
        {}^wY^j_k 
      &=
        (D_+A)^j_{k}
        +
        \frac{A_{k}^j + A_{k+1}^j}{2L(w^j_\cdot)}
        -
        \frac{1}{2L(w^j_\cdot)}
        ((D_0w)_{k+1}^j + (D_0w)_{k}^j)^2A_k^jA^j_{k+1}
    \end{align*}
  \end{prop}
  \begin{proof}
    Let 
    \[
      F(w_k^j)=\frac{(D_0w)_k^j}{L(w_\cdot^j)} + (D^2w)_k^j
    \]
    so that $(D_tw)_k^j = A_{k}F(w_k^j)$
    and
    \begin{align*}
        (D_tD_+w)^{j+1}_k
      &=
          \frac{1}{2}
          \left(
              D_+A_{k}
              \left(
                F(w_{k+1}^j)
                +
                F(w_{k}^j)
              \right)
              +
              D_+F(w_{k}^j)
              \left(
                A_{k}
                +
                A_{k+1}
              \right)
          \right).
    \end{align*}
    We can also compute that
    \begin{align*}
        D_+F(w^j_k)
      &=
        \frac{(D_0D_+w)_{k}^j}{L(w_\cdot^j)}
        +
        {(D^2D_+w)_k^j},\\
        (D^2w)_{k+1}^j
        +
        (D^2w)_{k}^j
      &=
        2(D_0D_+w)_{k}^j,\\
        (D_+A)_{k}
        \left(
          (D_0w)_{k+1}^j
          +
          (D_0w)_{k}^j
        \right)
      &= 
      -((D_0w)_{k}^j + (D_0w)_{k+1}^j)^2A_kA_{k+1}(D_0D_+w)_k^j.
    \end{align*}
    Plugging the three equalities above into our equation for
    $(D_tD_+w)^{j+1}_k$ gives the result.
  \end{proof}

  We now prove a maximum principle for the $D_+$ derivative
  which also gives us control over the $D_0$ derivative
  and therefore also the
  difference of both the non-local and the non-linear term.
  \begin{thm}
    \label{thm_inf_bound_d+w_numeric}
    If $w\in F(M,\R)$ satisfies
    equations
    \eqref{eq_full_scheme_2} and the initial and boundary values
    given by equations
    \eqref{eq_full_scheme_2_initial},
    \eqref{eq_full_scheme_2_boundary_n}, and
    \eqref{eq_num_boundary_condition}
    and $\delta u,\delta t$ are such that
    $(\delta u)^2 \geq 2\delta t$,
    and $\rho_0\geq \delta u(1 + \norm{(D_+w)_\cdot^0}_{L^\infty}^2)$
    then,
    for all $j\in T$,
    \[
        \norm{D_0w^{j+1}_\cdot}_{L^\infty}
      \leq
        \norm{D_+w^{j+1}_\cdot}_{L^\infty}
      \leq
          \norm{D_+w^j_{0}}_{L^\infty}.
    \]
  \end{thm}
  \begin{proof}
    We shall prove the result using an inductive argument.
    Let $j\in T$ and suppose that
    $\norm{(D_0w)^j_\cdot}_{L^\infty}\leq \norm{(D_+w)^0_\cdot}_{L^\infty}$.
    Since $(D_+w)_0^{j+1}=(D_+w)^j_0$,
    Proposition \ref{prop_DtD+_is},
    Lemma \ref{lem_max_inf_numerical_gen}, and our assumptions
    implies that for $k\in S\setminus\{0,n\}$ we have
    $\norm{(D_+w)_\cdot^{j+1}}_{L^\infty} 
      \leq \norm{D_+w^j_{\cdot}}_{L^\infty}$.
    Since, $2(D_0w)_k^j=(D_+w)^j_k + (D_+w)^j_{k-1}$
    we see that
    $\norm{(D_0w)_\cdot^{j+1}}_{L^\infty} 
      \leq \norm{D_+w^{j+1}_{\cdot}}_{L^\infty}$.
    Since
    $\norm{(D_0w)^0_\cdot}_{L^\infty}\leq \norm{(D_+w)^0_\cdot}_{L^\infty}$,
    by definition,
    our inductive assumption holds for $j=0$ and the result is true.
  \end{proof}

  We are now in a position to prove our two consistency results.
  \begin{thm}
    \label{thm_consistency_d+}
    Let $h$ be a solution to equation \eqref{eqgraph}
    in $C^4([0,\rho_0]\times[0,T))$ so that
    for all $t\in[0,T)$,
    $h(\rho_0, t)=0$ and $h_0(0,t)=0$.
    Fix $j\in T$.
    Let $v$ be the function, defined on the mesh
    $S\times\{(j)\delta t, (j+1)\delta t\}$,
    so that 
    $v_i^{j}=\pi(h)_i^{j}$
    and
    \begin{align*}
      v_0^{j+1}&=v_0^j + \delta t\,(D_t v)^j_1,\\
      v_i^{j+1}&= v_i^{j} + 
        \delta t\, F((D_0\pi(h))_i^{j}, (D^2\pi(h))_i^{j}, L[(\pi(h)_\cdot^{j})]),\\
      v_n^{j+1}&=0,
    \end{align*}
    where 
    \[
      F(x,y,z)=\frac{1}{1+x^2}\left(y + \frac{x}{z}\right).
    \]

    If there exists $B\in[0,\infty)$
    so that, for all $(x,j)\in [0,\rho_0]\times T$,
    \begin{align*}
      &\max\biggl\{
        \abs{h_{uu}(x, j\delta t)},
        \abs{h_{uuu}(x, j\delta t)},
        \abs{h_{uut}(x, j\delta t)},\\
      &\hspace{3cm}
        \abs{h_{utt}(x, j\delta t)}, 
        \abs{h_{ttt}(x, j\delta t)},
        \abs{h_{uuuu}(x, j\delta t)}
      \biggr\}\leq B,
    \end{align*}
    then, 
    \begin{align}
      \label{eq_d+_con_0}
      \frac{1}{\delta t}\abs{(D_+(\pi(h) - v))^{j+1}_0}
      &\leq 
        \delta u\,\frac{B}{2}
        +
        \delta t\, \frac{B}{2}
        +
        \frac{\delta t^2}{\delta u}\frac{B}{3},
    \end{align}
    and,
    for all $i\in S\setminus\{0,n\}$,
    \begin{align}
      \label{eq_d+_con_k}
      \frac{1}{\delta t}\abs{(D_+(\pi(h) - v))^{j+1}_i}
      &\leq
          \delta u
          \biggl(
            B
            \left(
              \frac{5}{6\rho_0}
              +
              \frac{1}{2}
            \right)
            +
            B^2
            \left(
              \frac{1}{\rho_0}
              +
              \frac{3}{2}\rho_0
              +
              \frac{17}{6}
            \right)
            +
            B^3
          \biggr)\\
        &\hspace{1cm}
          +
          (\delta u)^2
          \biggl(
            \frac{5}{6}B
            +
            B^2
            \left(
              \frac{5}{6\rho_0}
              +
              \frac{3}{4}\rho_0
              +\frac{21}{12}
            \right)
          \biggr)
          +
          B\frac{\delta t}{2}\nonumber\\
        &\hspace{1cm}
          +
          (\delta u)^3\frac{9}{24}B^2
          +
          \frac{1}{3}B\frac{\delta t^2}{\delta u}.\nonumber
    \end{align}
  \end{thm}
  \begin{proof}
    Our boundary condition implies that $h_{tu}(0,t)=0$
    thus repeated use of Taylor's theorem with remainder, and
    our assumptions implies that
    \begin{align*}
      \frac{1}{\delta t}\abs{(D_+(\pi(h) - v))^{j+1}_0}
      &\leq 
        \delta u\,\frac{B}{2}
        +
        \delta t\, \frac{B}{2}
        +
        \frac{\delta t^2}{\delta u}\frac{B}{3},
    \end{align*}
    as required.

    Now, let $i\in S\setminus\{0,n\}$.
    We have
    \begin{align*}
      \delta u\left(
        (D_+\pi(h))^{j+1}_i 
        - (D_+v)^{j+1}_i
        \right)
      &=
        \pi(h)^{j+1}_{i+1}
        -
        \pi(h)^{j+1}_i 
        -
        \pi(h)^{j}_{i+1}
        +
        \pi(h)^{j}_i\\
      &\hspace{-4cm}
        -\delta t 
        F((D_0\pi(h))_{i+1}^{j}, (D^2\pi(h))_{i+1}^{j}, L(\pi(h)_{\cdot}^{j}))
        +\delta t 
        F((D_0\pi(h))_i^{j}, (D^2\pi(h))_i^{j}, L(\pi(h)_\cdot^{j})),
    \end{align*}
    So that, by assumption,
    \begin{align*}
        \frac{1}{\delta t}\abs{(D_+(\pi(h) -v))^{j+1}_i} 
      &\leq
        \abs{E_1}
        +
        \abs{E_2}
        +
        \abs{E_3}
        +
        B\frac{\delta u}{2}
        +
        B\frac{\delta t}{2}
        +
        \frac{1}{3}B\frac{\delta t^2}{\delta u},
    \end{align*}
    where, with implicit evaluation of all functions
    at $(i\delta u, j\delta t)$,
    \begin{align*}
        E_1 &= F_1\biggl(h_u, h_u, h_{uu}, L[h(\cdot,j\delta t)]\biggr)\\
      &\hspace{1cm}
          -
        F_1\biggl((D_0\pi(h))_{i+1}^j, (D_0\pi(h))_{i}^j, 
          (D_0D_+\pi(h))^j_i, L[\pi(h)^{j}_\cdot]\biggr),\\
      E_2&=
        F_2\biggl(h_u, h_u, h_{uu}, h_{uu}\biggr)\\
      &\hspace{1cm}
        -
        F_2\biggl((D_0\pi(h))_{i+1}^j, (D_0\pi(h))_{i}^j,
          (D_0D_+\pi(h))^j_i, (D^2\pi(h))_{i+1}^j\biggr),\\
      E_3&=
        F_3\biggl(h_u, h_{uuu}\biggr)
        -
        F_3\biggl((D_0\pi(h))_{i+1}^j,(D^2D_+\pi(h))^j_i\biggr),
    \end{align*}
    and
    \begin{gather*}
      F_1(x_1,x_2,y,z) = \frac{y}{z}\frac{1-x_1x_2}{(1+x_1^2)(1+x_2^2)},
      \quad\quad
      F_2(x_1,x_2,y_1, y_2) = y_1y_2\frac{x_1 + x_2}{(1+x_1^2)(1+x_2^2)},\\
      F_3(x,y)  = \frac{y}{1+x^2}.
    \end{gather*}

    We continue by analysing $E_1$.
    Since
    \[
      \abs{\frac{1 - x_1x_2}{(1+x_1^2)(1+x_2^2)}
      -
      \frac{1 - a_1x_2}{(1+a_1^2)(1+x_2^2)}}
      \leq \abs{a_1 - x_1}
    \]
    and as $F_1(x_1,x_2,y,z)=F_1(x_2,x_1,y,z)$
    we have, by assumption, and repeated use of Taylor's theorem
    with remainder that
    \begin{align*}
          \abs{E_1}
      &\leq
          \delta u
          \biggl(
            \frac{5}{6}\frac{B}{\rho_0}
            +
            B^2
            \left(
              \frac{1}{\rho_0}
              +\frac{3}{2}\rho_0
            \right)
          \biggr)
          +
          (\delta u)^2
          \biggl(
            \frac{5}{6}\frac{B^2}{\rho_0}
            +
            \frac{3}{4}\rho_0 B^2
          \biggr).
    \end{align*}
    We now study the $E_2$ term.
    As
    \[
      \abs{\frac{x_1 + x_2}{(1+ x_1^2)(1+x_2^2)}
      -
      \frac{a_1 + x_2}{(1+ a_1^2)(1+x_2^2)}}
      \leq
      \abs{a_1 - x_1}
    \]
    and as $F_2$ is symmetric in its first two arguments
    we have, using similar logic as above,
    \begin{align*}
        \abs{E_2}
      &\leq
        \delta u B^2
        \biggl(
          B 
          + 
          \frac{11}{6}
        \biggr)
        +
        (\delta u)^2
        B^2
        \biggl(
          \frac{13}{12} 
          + 
          \frac{5}{6}B
        \biggr)
        +
        (\delta u)^3\frac{9}{24}B^2
    \end{align*}
    We now study the $E_3$ term.
    As
    \[
      \abs{\frac{1}{1+x^2}
      -
      \frac{1}{1+a^2}}
      =
      \abs{a-x}
    \]
    we have, similarly to above,
    that
    \[
      \abs{E_3}
      \leq 
        \delta u B(1 + B)
        +
        \frac{2}{3}(\delta u)^2 B^2.
    \]
    Therefore the result holds.
  \end{proof}

  \begin{thm}
    \label{thm_consistency}
    Let $h$ be a solution to equation \eqref{eqgraph}
    in $C^4([0,\rho_0]\times[0,T))$ so that
    for all $t\in[0,T)$,
    $h(\rho_0, t)=0$ and $h_0(0,t)=0$.
    Fix $j\in T$.
    Let $v$ be the function, defined on the mesh
    $S\times\{j\delta t, (j+1)\delta t\}$,
    so that 
    $v_i^{j}=\pi(h)_i^{j}$
    and
    \begin{align*}
      v_0^{j+1}&=v_0^j + \delta t\,(D_t v)^j_1,\\
      v_i^{j+1}&= v_i^{j} + 
        \delta t\, F((D_0\pi(h))_i^{j}, (D^2\pi(h))_i^{j}, L[(\pi(h)_\cdot^{j})]),\\
      v_n^{j+1}&=0,
    \end{align*}
    where $F$ is as defined in Theorem \ref{thm_consistency_d+}.

    If there exists $B\in\R^+$ so that, for all 
    $x\in [0,\rho_0], j\in T\setminus\{0\}$,
    \[
      \max\left\{
        \abs{h_{uu}(x, j\delta t)},
        \abs{h_{uuu}(x, j\delta t)},
        \abs{h_{uuuu}(x, j\delta t)},
        \abs{h_{tt}(x, j\delta t)}
      \right\}\leq B
    \]
    then, 
    \begin{align}
      \label{eq_con_0}
      \frac{1}{\delta t}\abs{\pi(h)^{j+1}_0 - v^{j+1}_0}
      &\leq
        \delta u B
        \left(
          1
          +
          B
          +
          \frac{5}{2\rho_0}
        \right)
        +
        (\delta u)^2
        B
        \left(
          \frac{2}{3}B
          +
          \frac{2}{3\rho_0}
          +
          \frac{3}{4}
        \right)
        +
        \delta t\frac{B}{2},
    \end{align}
    $\abs{\pi(h)_n^j}=0$,
    and,
    for all $i\in S\setminus\{0,n\}$,
    \begin{align}
      \label{eq_con_k}
      \frac{1}{\delta t}\abs{\pi(h)^{j+1}_i - v^{j+1}_i}
      &\leq
        \delta u\frac{3}{2\rho_0}B
        +
        (\delta u)^2
        B
        \left(
          \frac{1}{12}
          +
          \frac{B}{6}
          +
          \frac{1}{6\rho_0}
        \right)
        +
        \delta t\frac{B}{2}.
    \end{align}
  \end{thm}
  \begin{proof}
    By definition $h(n\delta u, j\delta t)=0$ for all
    $j\in T$. Therefore  $\abs{\pi(h)_n^j}=0$.
    With an application of Taylor's theorem with remainder,
    for $i\in S\setminus\{0,n\}$,
    \begin{align*}
      \frac{1}{\delta t}\biggl(
        \pi(h)^{j+1}_i& - v^{j+1}_i
      \biggr)
      =
        \frac{\delta t}{2} h_{tt}(i\delta u, (j+\tau)\delta t)
        +
        \sum_{a=1}^3 E_a,
    \end{align*}
    where $\tau\in(0,1)$ and, with implicit evaluation of functions
    at $(i\delta u, j\delta t)$,
    \begin{align*}
      E_1 &= 
        F(h_u, h_{uu}, L[h(\cdot, j\delta t)])
        -
        F((D_0\pi(h))_i^j, h_{uu}, L[h(\cdot, j\delta t)]),\\
      E_2 &= 
        F((D_0\pi(h))_i^j, h_{uu}, L[h(\cdot, j\delta t)])
        -
        F((D_0\pi(h))_i^j, (D^2\pi(h))_i^j, L[h(\cdot, j\delta t)]),\\
      E_3 &= 
        F((D_0\pi(h))_i^j, (D^2\pi(h))_i^j, L[h(\cdot, j\delta t)])
        -
        F((D_0\pi(h))_i^j, (D^2\pi(h))_i^j, L[(\pi(h))_\cdot^j]).
    \end{align*}
    Using an analysis similar to that in the proof
    of Theorem \ref{thm_consistency_d+}
    and Lemma \ref{lem_length_diff_bound}
    gives,
    for $i\in S\setminus\{0,n\}$,
    \begin{align*}
          \frac{1}{\delta t}\abs{\pi(h)^{j+1}_i - v^{j+1}_i}
      \leq
        \delta u\frac{3}{2\rho_0}B
        +
        (\delta u)^2
        B
        \left(
          \frac{1}{12}
          +
          \frac{B}{6}
          +
          \frac{1}{6\rho_0}
        \right)
        +
        \delta t\frac{B}{2}
    \end{align*}
    as required.

    We now prove the claim for the left hand ($i=0$) boundary.
    Taylor's theorem with remainder gives
    \begin{align*}
      \frac{1}{\delta t}\left(
        \pi(h)^{j+1}_0 - v^{j+1}_0
      \right)
      &=
        \frac{\delta t}{2} h_{tt}(0, (j+\tau)\delta t)
        +
        \sum_{a=1}^3 E_a,
    \end{align*}
    where $\tau\in(0,1)$ and
    \begin{align*}
      E_1 &= 
        F(h_u(0,j\delta t), h_{uu}(0,j\delta t), L[h(\cdot, j\delta t)])
        -
        F((D_0\pi(h))_1^j, h_{uu}(0,j\delta t), L[h(\cdot, j\delta t)])\\
      E_2 &= 
        F((D_0\pi(h))_1^j, h_{uu}(0,j\delta t), L[h(\cdot, j\delta t)])
        -
        F((D_0\pi(h))_1^j, (D^2\pi(h))_1^j, L[h(\cdot, j\delta t)])\\
      E_3 &= 
        F((D_0\pi(h))_1^j, (D^2\pi(h))_1^j, L[h(\cdot, j\delta t)])
        -
        F((D_0\pi(h))_1^j, (D^2\pi(h))_1^j, L[(\pi(h))_\cdot^j]).
    \end{align*}
    Therefore, our assumptions,
    Lemma
    \ref{lem_length_diff_bound},
    and
    repeated use of Taylor's theorem with remainder gives
    \begin{align*}
          \frac{1}{\delta t}\abs{\pi(h)^{j+1}_0 - v^j_0}
      &\leq
        \delta u B
        \left(
          1
          +
          B
          +
          \frac{5}{2\rho_0}
        \right)
        +
        (\delta u)^2
        B
        \left(
          \frac{2}{3}B
          +
          \frac{2}{3\rho_0}
          +
          \frac{3}{4}
        \right)
        +
        \delta t\frac{B}{2}
    \end{align*}
    as required.
  \end{proof}

  Rather than prove stability and then show that stability and
  consistency imply convergence, we exploit the known bounds
  on the second derivative of solutions to the modelled PDE
  to prove convergence directly. Just as for consistency we
  require two convergence results, one for the forward spatial
  derivative and one for solutions.

  \begin{lem}
    \label{lem_bounds_on_ratios}
    There exists $K\in(0,\infty)$ so that
    for all $w,x,y,z\in\R$ the following inequalities hold
    \begin{gather*}
      \abs{\frac{{x + y}}{(1+x^2)(1+y^2)}}\leq K,\\[1ex]
      \abs{\frac{-1 + w y + z(w + y) - z^2 + w y z^2 + x^2 z (w + y)}%
        {(1+w^2)(1+x^2)(1+y^2)(1+z^2)}}\leq K\\[1ex]
      \abs{\frac{x + z + 2w - w^2y^2(x + z) + 2w^2y - 2wxz - 2w^2xyz}%
        {(1+w^2)(1+x^2)(1+y^2)(1+z^2)}}\leq K\\[1ex]
    \end{gather*}
  \end{lem}
  \begin{proof}
    Outside of a large enough compact set all three ratios of polynomials
    will be less than $1$. Since the denominators are never zero the
    ratios are continuous and hence the result holds.
  \end{proof}

  \begin{thm}
    \label{thm_D_+_diff_bound}
    Let $h$ be a solution to equation \eqref{eqgraph}
    in $C^4([0,\rho_0]\times[0,T))$ so that
    for all $t\in[0,T)$,
    $h(\rho_0, t)=0$ and $h_0(0,t)=0$.
    Let $w\in F(M,\R)$ be a solution to 
    the equation \eqref{eq_full_scheme_2}
    and the initial and boundary
    equations
    \eqref{eq_full_scheme_2_initial},
    \eqref{eq_full_scheme_2_boundary_n}, and
    \eqref{eq_num_boundary_condition}.

    If $h$ satisfies the conditions of Theorems
    \ref{thm_consistency_d+} involving the constant $B$
    and if $\delta t, \delta u$
    satisfy the conditions of Theorem \ref{thm_inf_bound_d+w_numeric}
    then, 
    \begin{align*}
      \abs{(D_+(\pi(h) - w))_0^{j+1}}
      \leq \delta t\delta u\, C_1 +
        \abs{(D_+(\pi(h) - w))_0^{j}}
    \end{align*}
    where $\delta u\, C_1$ is the right hand side of
    equation \eqref{eq_d+_con_0} 
    and,
    for $k\in S\setminus\{0,n\}$,
    \begin{align*}
      \abs{D_+(v-w)_k^{j+1}}
        &\leq 
          \delta t\,
          \delta u\,
          C_2
          +
          \left(
            1 + 3 K B\delta t
            \left(
              \frac{1}{2}
              + 
              B
              +
              \frac{1}{\rho_0}
            \right)
          \right)
          \norm{D_+(v-w)^j_\cdot}_{L^\infty}
    \end{align*}
    where $\delta u\, C_2$
    is the right hand side of equation
    \eqref{eq_d+_con_k}, and
    where $K$ is the constant defined in
    Lemma \ref{lem_bounds_on_ratios}.
    In particular,
    \begin{align}
      \label{eq_d+_norm_global_bound}
        &\norm{(D_+(\pi(h) - w))^{j+1}_\cdot}_{L^\infty}
      \leq\\
        &\hspace{2cm}
        \delta u\frac{C_3}{C_4}
        \left(\exp\left(C_4 T\right) -1\right)
        +
        \exp\left(
          C_4 T
        \right)
        \norm{(D_+(\pi(h) - w))^{0}_\cdot}_{L^\infty},\nonumber
    \end{align}
    where $C_3=\max\{C_1,C_2\}$ and
    \[
      C_4 = 3 K B
            \left(
              \frac{1}{2}
              + 
              B
              +
              \frac{1}{\rho_0}
            \right)
    \]
  \end{thm}
  \begin{proof}
    Fix $j\in T\setminus\{m\}$.
    Let $v$ be the function, defined on the mesh
    $S\times\{(j)\delta t, (j+1)\delta t\}$,
    so that 
    $v_i^{j}=\pi(h)_i^{j}$
    and
    \begin{align*}
      v_0^{j+1}&=v_0^j + \delta t\,(D_t v)^j_1,\\
      v_i^{j+1}&= v_i^{j} + 
        \delta t F((D_0\pi(h))_i^{j}, (D^2\pi(h))_i^{j}, L[(\pi(h)_\cdot^{j})]),\\
      v_n^{j+1}&=0.
    \end{align*}

    For all $k\in S\setminus\{n\}$,
    we have
    \begin{align*}
      \abs{D_+(\pi(h) - w_k)^{j+1}}
      \leq
      \abs{D_+(\pi(h) - v)^{j+1}_k} 
      + 
      \abs{D_+(v - w)_k^{j+1}}.
    \end{align*}
    Theorem \ref{thm_consistency_d+} allows us to control the first term on the
    right hand side.
    By definition of $v$ and equation \eqref{eq_num_boundary_condition}
    we can compute that
    $(D_+(v-w))^{j+1}_0 = (D_+(v-w))^{j}_0$.
    This gives us the claimed result for $k=0$.

    We shall now estimate the $\abs{D_+(v - w)_k^{j+1}}$
    term in a manor 
    similar to the proof of Lemma \ref{lem_max_inf_numerical_gen}.
    By Proposition \ref{prop_DtD+_is}
    we know that, for all $k\in S\setminus\{0,n\}$,
    \begin{align*}
        {(D_tD_+(v - w))_k^{j}}
      &=
        \left(
          {}^v\!X_k
          -
          {}^w\!X_k
        \right)
        (D^2D_+v)^j_k
        +
        {}^w\!X_k
        (D^2D_+(v-w))^j_k\\
      &\hspace{1cm}
        +
        \left(
          {}^vY_k
          -
          {}^wY_k
        \right)
        (D_0D_+v)^j_k
        +
        {}^wY_k
        (D_0D_+(v-w))^j_k.
    \end{align*}
    where ${}^v\!X^j_k, {}^vY^j_k$ and
    ${}^w\!X^j_k, {}^wY^j_k$ are defined in
    Proposition \ref{prop_DtD+_is}.
    We can compute that
    \begin{align*}
        {}^v\!X_k
        -
        {}^w\!X_k
      &=
          -
          \frac{1}{2}
          \left(
            (D_0w)_k^j + (D_0v)_k^j
          \right)
          {}^w\!A_k {}^v\!A_k
          (D_0(v-w))_k^j\\
        &\hspace{1cm}
          -
          \frac{1}{2}
          \left(
            (D_0w)_{k+1}^j + (D_0v)_{k+1}^j
          \right)
          {}^w\!A_{k+1} {}^v\!A_{k+1}
          (D_0(v-w))_{k+1}^j.
    \end{align*}
    We also know that
    \begin{align*}
        {}^vY^j_k 
        -
        {}^wY^j_k 
      &=
          {\left(D_+\left({}^v\!A - {}^w\!A\right)\right)_{k}}\\
        &\hspace{1cm}
          +
          \frac{1}{2L(w^j_\cdot)}
          \left(
            {{}^v\!A_{k} - {}^w\!A_{k}}
            +
            {{}^v\!A_{k+1} - {}^w\!A_{k+1}}
          \right)\\
        &\hspace{-2cm}
          -
          \frac{1}{2L(w^j_\cdot)}
          \left(
            ((D_0v)_{k}^j + (D_0v)_{k+1}^j)^2\,{}^v\!A_k\,{}^v\!A_{k+1}
            -
            ((D_0w)_{k}^j + (D_0w)_{k+1}^j)^2\,{}^w\!A_k\,{}^w\!A_{k+1}
          \right).
    \end{align*}
    We will
    rewrite the three terms on the right hand side of this
    equation.

    We begin with the first term.
    Some algebra shows that
    \begin{align*}
          (D_+{}^v\!A)_k - (D_+{}^w\!A)_{k}
      &=\\
      &\hspace{-2cm}
          {}^w\!A_{k+1}
          {}^w\!A_{k}
          {}^v\!A_{k+1} 
          {}^v\!A_{k} 
          \biggl(
            A_1 D_0(v-w)_k + A_2 D_0(v-w)_{k+1}
          \biggr)
          D_+D_0v_k\\
        &\hspace{1cm}
          -
          (D_0w_k + D_0w_{k+1})
          {}^w\!A_{k+1}
          {}^w\!A_{k}\,
          D_+D_0(v-w)_k,
    \end{align*}
    where $A_1$ and $A_2$ are such that
    $\abs{A_1 {}^w\!A_{k+1} {}^w\!A_{k} {}^v\!A_{k+1} {}^v\!A_{k}}\leq K$,
    $\abs{A_2 {}^w\!A_{k+1} {}^w\!A_{k} {}^v\!A_{k+1} {}^v\!A_{k}}\leq K$
    where $K$ is the constant given in Lemma
    \ref{lem_bounds_on_ratios}.
    For the second term we have
    \begin{align*}
        \frac{1}{2L(w^j_\cdot)}
        &\left(
          {{}^v\!A_{k} - {}^w\!A_{k}}
          +
          {{}^v\!A_{k+1} - {}^w\!A_{k+1}}
        \right)\\
      &=
          -
          \frac{1}{2L(w^j_\cdot)}
          \left(
            (D_0w)_k^j + (D_0v)_k^j
          \right)
          {}^w\!A_k\,{}^v\!A_k
          (D_0(v-w))_k^j\\
        &\hspace{1cm}
          -
          \frac{1}{2L(w^j_\cdot)}
          \left(
            (D_0w)_{k+1}^j + (D_0v)_{k+1}^j
          \right)
          {}^w\!A_{k+1}\,{}^v\!A_{k+1}
          (D_0(v-w))_{k+1}^j.
    \end{align*}

    For the third term we can compute that
    \begin{align*}
          &-\frac{1}{2L(w^j_\cdot)}
          \biggl(
            ((D_0v)_{k}^j + (D_0v)_{k+1}^j)^2\,{}^v\!A_k\,{}^v\!A_{k+1}
            -
            ((D_0w)_{k}^j + (D_0w)_{k+1}^j)^2\,{}^w\!A_k\,{}^w\!A_{k+1}
          \biggr)\\
      &\hspace{1cm}= 
        -
        \frac{1}{2L(w^j_\cdot)}
        {}^v\!A_k\,{}^v\!A_{k+1}{}^w\!A_k\,{}^w\!A_{k+1}
        \biggl(
          Q_1 D_0(v-w)_k + Q_{2} D_0(v-w)_{k+1}
        \biggr),
    \end{align*}
    where $Q_1,Q_2$ are such that
    $\abs{Q_1 {}^v\!A_k\,{}^v\!A_{k+1}{}^w\!A_k\,{}^w\!A_{k+1}}\leq K$,
    and
    $\abs{Q_2 {}^v\!A_k\,{}^v\!A_{k+1}{}^w\!A_k\,{}^w\!A_{k+1}}\leq K$,
    where $K$ is the constant given in Lemma
    \ref{lem_bounds_on_ratios}.

    If one plugs these computations into the
    original equality for
    ${(D_tD_+(v - w))_k^{j}}$ and uses the equation
    $2D_0(v-w)^j_k=D_+(v-w)_k^j+D_+(v-w)^j_{k-1}$
    then we get
    \begin{align*}
        {(D_tD_+(v - w))_k^{j}}
      &=
        (D_+(v-w))^j_{k+1}\biggl(
          \frac{1}{(\delta u)^2}{}^w\!X_k
          +
          \frac{1}{2 \delta u}{}^wY_k\\
        &\hspace{2cm}
          -
          \frac{1}{2\delta u}
          (D_0w_k + D_0w_{k+1})
          {}^w\!A_{k+1}
          {}^w\!A_{k}
          (D_0D_+v)^j_k
          + B_1
        \biggr)\\
      &+
        (D_+(v-w))^j_{k+1}\biggl(
          -\frac{2}{(\delta u)^2}{}^w\!X_k + B_2
        \biggr)\\
      &+
        (D_+(v-w))^j_{k-1}
        \biggl(
          \frac{1}{(\delta u)^2}{}^w\!X_k
          -
          \frac{1}{2\delta u}
          {}^wY_k\\
        &\hspace{2cm}
          +
          \frac{1}{2\delta u}
          (D_0w_k + D_0w_{k+1})
          {}^w\!A_{k+1}
          {}^w\!A_{k}
          (D_0D_+v)^j_k
          +B_3
        \biggr),
    \end{align*}
    where, 
    by Lemma \ref{lem_bounds_on_ratios},
    our assumptions about $B$,
    and repeated use of Taylor's theorem with remainder,
    we know that
    \[
        \max\{\abs{B_1}, \abs{B_2}, \abs{B_3}\}
      \leq
        K B
        \left(
          \frac{1}{2}
          + 
          B
          +
          \frac{1}{L(w^j\cdot)}
        \right).
    \]

    As in the proof of Theorem \ref{thm_inf_bound_d+w_numeric},
    Proposition \ref{prop_DtD+_is},
    Lemma \ref{lem_max_inf_numerical_gen}, and our assumptions,
    imply that
    \begin{align*}
        \abs{(D_+(v - w))_k^{j+1}}
      &=
        \norm{(D_+(v-w))^j_{\cdot}}_{L^\infty}\biggl(
          1
          + \delta t(\abs{B_1} + \abs{B_2}+\abs{B_3})
        \biggr)\\
      &\leq
        \norm{(D_+(v-w))^j_{\cdot}}_{L^\infty}\biggl(
          1
          + 3 K B \delta t
          \left(\frac{1}{2} + B + \frac{1}{\rho_0}\right)
        \biggr),
    \end{align*}
    as required.

    The last inequality follows by telescoping the right hand side
    of the bound for
    $\abs{(D_+(v-w)_k^{j+1}}$ and finding a bound for the resulting
    truncated geometric series.
  \end{proof}

  \begin{thm}
    \label{thm_convergence}
    Let $h$ be a solution to equation \eqref{eqgraph}
    in $C^4([0,\rho_0]\times[0,T))$ so that
    for all $t\in[0,T)$,
    $h(\rho_0, t)=0$ and $h_0(0,t)=0$.
    Let $w\in F(M,\R)$ be a solution to 
    the equation \eqref{eq_full_scheme_2}
    and the initial and boundary
    equations
    \eqref{eq_full_scheme_2_initial},
    \eqref{eq_full_scheme_2_boundary_n}, and
    \eqref{eq_num_boundary_condition}.

    If $h,\delta u,\delta t$ 
    satisfy the conditions of Theorems \ref{thm_consistency}
    and \ref{thm_D_+_diff_bound},
    with the same constant bound $B$ for all the specified derivatives
    of $h$, 
    and
    \begin{align*}
         1 - \frac{1}{1 + \norm{(D_0w)^0_\cdot}_{L^\infty}^2}
      &\geq
          \delta u,\\
          \frac{1}{1 + \norm{(D_0w)^0_\cdot}_{L^\infty}^2}
      &\geq
          \frac{\delta u}{2}\frac{1}{\rho_0}
          +
          \frac{\delta u}{2}
          \left(
            \frac{1}{\rho_0}\left(1+\frac{(\delta u)^2}{6}\right)B
            +
            \left(1 + \frac{(\delta u)^2}{12}\right)B
          \right)
    \end{align*}
    then, for all $k\in S\setminus\{0\}$,
    \begin{align*}
        \abs{\pi(h)^{j+1}_k - w_k^{j+1}}
      &\leq
          \delta t\delta u\,C
          + 
          \norm{\pi(h)^{j}_\cdot - w_\cdot^{j}}_{L^\infty}\\
        &\hspace{1cm}
          +
          \delta t
          \frac{1}{\rho_0}
          \norm{(D_+(\pi(h)- w))_\cdot^j}_{L^\infty}.
    \end{align*}
    where $\delta u\, C$ is the
    right hand side of
    equation \eqref{eq_con_0} if $k= 0$ and
    is the right hand side of equation
    equation \eqref{eq_con_k} if $k\neq 0$.

    In particular,
    \begin{align}
      \label{eq_conv_global_bound}
        \norm{\pi(h)^{j+1}_\cdot - w_\cdot^{j+1}}_{L^\infty}
      &\leq
          \norm{\pi(h)^{0}_\cdot - w_\cdot^{0}}_{L^\infty}
          +
          \delta u\, T C\\
        &\hspace{1cm}
          +
          \delta u\,
          \frac{T}{\rho_0}
          \frac{C_3}{C_4}
          \left(\exp\left(C_4 T\right) -1\right)\nonumber\\
        &\hspace{1cm}
          +
          \frac{T}{\rho_0}
          \exp\left(
            C_4 T
          \right)
          \norm{(D_+(\pi(h) - w))^{0}_\cdot}_{L^\infty},\nonumber
    \end{align}
    where $C_3,C_4$ are defined in Theorem \eqref{thm_D_+_diff_bound}
    and $C=\max\{C_1,C_2\}$
    where
    $\delta u\, C_1$ is the right hand side of 
    equation \eqref{eq_con_0}
    and $\delta t\,C_2$ is the right hand side
    of equation \eqref{eq_con_k}.
  \end{thm}
  \begin{proof}
    Fix $j\in T\setminus\{m\}$.
    Let $v$ be the function, defined on the mesh
    $S\times\{(j)\delta t, (j+1)\delta t\}$,
    so that 
    $v_i^{j}=\pi(h)_i^{j}$
    and
    \begin{align*}
      v_0^{j+1}&=v_0^j + \delta t\,(D_t v)^j_1,\\
      v_i^{j+1}&= v_i^{j} + 
        \delta t F((D_0\pi(h))_i^{j}, (D^2\pi(h))_i^{j}, L[(\pi(h)_\cdot^{j})]),\\
      v_n^{j+1}&=0.
    \end{align*}

    We have, for $k\in S\setminus\{n\}$,
    \begin{align*}
      \abs{\pi(h)^{j+1}_k - w_k^{j+1}}
      \leq
      \abs{\pi(h)^{j+1}_k - v^{j+1}_k} 
      + 
      \abs{v^{j+1}_k - w_k^{j+1}}.
    \end{align*}
    The first term on the right hand side is bounded by $\delta t\delta u C$
    for some 
    $C\in(0,\infty)$, Theorem \ref{thm_consistency}.
    We shall now estimate the $\abs{v^{j+1}_k - w_k^{j+1}}$ term.
    We have, for $k\in S\setminus\{0,n\}$, by direct computation
    \begin{align*}
      \label{eq_conv_k_dt}
          D_t(v^{j}_k - w_k^{j})
      &=
          X_k^j 
          (D^2(v-w))^j_k
          +
          Y_k^j
          (D_0(v-w))^j_k
          +
          {}^v\!A_{k}(D_0v)^j_k
          \left(
            \frac{1}{L(v^{j}_\cdot)}
            -
            \frac{1}{L(w^{j}_\cdot)}
          \right),
    \end{align*}
    where
    \begin{align*}
      X_k^j &= {}^w\!A_k\\
      Y^j_k &= 
        \frac{{}^w\!A_{k}}{L(w^{j}_\cdot)}
        -
        \left(
          \frac{(D_0v)^j_k}{L(w^{j}_\cdot)}
          +
          (D^2v)^j_k
        \right)
        \frac{(D_0w)_k^j + (D_0v)_k^j}{(1+((D_0w)_k^j)^2)(1+((D_0v)_k^j)^2)}
    \end{align*}

    Since $0\leq X_k^j\leq 1$ 
    and $\delta t/(\delta u)^2\leq 1/2$
    we know
    that
    \[
      1 - 2 X_k^j\frac{\delta t}{(\delta u)^2}\geq 0.
    \]
    Two applications of Taylor's theorem with remainder, and
    our assumptions,
    imply that
    \[
      \abs{Y_k^j}
      \leq
      \frac{1}{\rho_0}
      +
      \left(
        \frac{1}{\rho_0}\left(1+\frac{(\delta u)^2}{6}\right)B
        +
        \left(1 + \frac{(\delta u)^2}{12}\right)B
      \right).
    \]
    Therefore, by assumption,
    $X^j_k\pm Y^j_k\delta u/2\geq 0$.
    Using an argument similar to the proof
    of Lemma \ref{lem_max_inf_numerical_gen}
    we get, by Lemma \ref{lem_L_dif_numer}, that
    \begin{align*}
        \abs{v^{j+1}_k - w_k^{j+1}}
      &\leq
        \norm{\pi(h)^{j}_\cdot - w_\cdot^{j}}_{L^\infty}
        +
        \delta t
        \frac{1}{\rho_0}
        \norm{(D_+(\pi(h)- w))_\cdot^j}_{L^\infty}.
    \end{align*}

    We now turn to the estimation of
    the 
    $\abs{v^{j+1}_0 - w_0^{j+1}}$ term.
    Suppose that 
    $0\leq v^{j+1}_0 - w_0^{j+1}$, then
    \begin{align*}
        0
      \leq
        v^{j+1}_0 - w_0^{j+1}
      \leq
        v^{j}_0 - w_0^{j}
        +\delta t (D_t(v-w))_1^j.
    \end{align*}
    Therefore if
    $ v^{j}_0 - w_0^{j} \leq v^{j}_1 - w_1^{j}$
    we get the same result as for 
    $v^{j+1}_1 - w_1^{j+1}$. Suppose instead, then,
    that
    $v^{j}_0 - w_0^{j}> v^{j}_1 - w_1^{j}$. Then we have,
    by expanding $D_t(v-w)$ as above,
    \begin{align*}
        0
      &\leq
        v^{j+1}_0 - w_0^{j+1}
      \leq
        \left(v^{j}_0 - w_0^{j}\right)
        \biggl(
          1 
          - 
          X_1^j\frac{\delta t}{(\delta u)^2}
          +
          Y_1^j\frac{\delta t}{\delta u}
        \biggr)\\
      &\hspace{1cm}
        +
        \left(v^{j}_2 - w_2^{j}\right)
        \biggl(
          X_1^j\frac{\delta t}{(\delta u)^2}
          +
          Y_1^j\frac{\delta t}{\delta u}
        \biggr)
        +
        {}^v\!A_1(D_0v)_1^j
        \left(
          \frac{1}{L(v^j_\cdot)}
          -
          \frac{1}{L(w^j_\cdot)}
        \right).
    \end{align*}
    By reasoning similar to above, the coefficients of the 
    first two terms on the right hand side must be positive and
    hence we get
    \begin{align*}
        \abs{v^{j+1}_0 - w_0^{j+1}}
      &\leq
        \norm{(\pi(h) - w_k)^{j}_\cdot}_{L^\infty}
        +
        \delta t
        \frac{1}{\rho_0}
        \norm{(D_+(\pi(h)- w))_\cdot^j}_{L^\infty}
    \end{align*}
    as claimed.
    If instead
    $0> v^{j+1}_0 - w_0^{j+1}$, then we write
    \begin{align*}
        0
      \leq
        w^{j+1}_0 - v_0^{j+1}
      \leq
        w^{j}_0 - v_0^{j}
        +\delta t (D_t(w-v))_1^j
    \end{align*}
    and we get the same bound on
    $\abs{v^{j+1}_0 - w_0^{j+1}}$ using similar logic.

    It remains to prove 
    the last inequality.
    Substitution of equation \eqref{eq_d+_norm_global_bound}
    in to our recurrence for $v-w$
    results in a simple recursive inequality.
    Telescoping with respect to $j$
    gives the result, once we recall that $m\delta t = T$.
  \end{proof}

\begin{rmk}
Note that $(D_0w)^0$ can not vanish, since we assumed that the curvature of the initial smooth data is strictly positive, zero at precisely one point, and then negative until the curve reaches the axis.
This means in particular that the initial profile can not be a constant.
If $D_0w_0$ is zero, then the initial profile is constant, and the height at the reflection axis must be zero.
\end{rmk}
  
  We have proven the existence of solutions,
  Theorem \ref{thm_inf_bound_w_numeric}, to our numerical scheme.
  We have also proven that our scheme is consistent,
  Theorems \ref{thm_consistency_d+}
  and \ref{thm_consistency},
  and convergent, Theorems \ref{thm_D_+_diff_bound}
  and
  \ref{thm_convergence}.
We have, however, yet to prove that our numeric scheme is well-posed; that is, that the scheme is stable.
We do this now. 

  \begin{thm}
    Let $w,v\in F(M,\R)$ be two solutions to 
    the equation \eqref{eq_full_scheme_2}
    and the boundary
    equations
    \eqref{eq_full_scheme_2_boundary_n}, and
    \eqref{eq_num_boundary_condition}.

    If both $\delta t, \delta u$ satisfy the conditions of
    Theorems \ref{thm_D_+_diff_bound} and
    \ref{thm_convergence}
    then, for all $k\in S$ and all $j\in T\setminus\{m\}$,
    \begin{align*}
          \norm{w^{j}_\cdot-v^j_\cdot}_{L^\infty}
      \hspace{-1cm}&\hspace{1cm}
          \leq
            \delta u\, 2 T C\\
        &
          +
          \delta u
          \frac{2T}{\rho_0}
          \frac{C_3}{C_4}
          \left(\exp\left(C_4 T\right) -1\right)\\
        &
          +
          \norm{w^{0}_\cdot - v_\cdot^{0}}_{L^\infty}
          +
          \frac{T}{\rho_0}
          \exp\left(
            C_4 T
          \right)
          \norm{(D_+(w - v))^{0}_\cdot}_{L^\infty},
    \end{align*}
    where $C, C_3, C_4$ are defined in Theorem
    \ref{thm_convergence},
  \end{thm}
  \begin{proof}
    Let $h$ be a solution to equation \eqref{eqgraph}
    in $C^4([0,\rho_0]\times[0,T))$ so that
    for all $t\in[0,T)$,
    $h(\rho_0, t)=0$, $h_0(0,t)=0$,
    and for all $i\in S$, $h(i\delta u, 0)=w^0_i$.
    Since we specify the initial conditions of
    $h$ at $n$ points we can assume that
    $u\mapsto h(u,0)$ is in $W^{K,P}([0,\rho_0])$
    for any $K\geq 3, P\geq 2$. Thus by
    Proposition \ref{propwkp}
    the conditions on $h$ given in Theorem
    \ref{thm_convergence} hold.
    
    Thus we can estimate
    \begin{align*}
        \norm{w^{j}_\cdot-v^j_\cdot}_{L^\infty}
      &\leq
        \norm{w^{j}_\cdot-\pi(h)^j_\cdot}_{L^\infty}
        +
        \norm{\pi(h)^{j}_\cdot-v^j_\cdot}_{L^\infty}
    \end{align*}
    and applying Theorem \ref{thm_convergence},
    specifically equation \eqref{eq_conv_global_bound},
    to both terms on the right hand of the above
    inequality.
    Noting that $h$ is equal to $w$ on the initial boundary
    gives us the required result.
  \end{proof}

\section{Numerical results}
  This section describes the numerical results of implementing the scheme
  described in Section \ref{sec_numerical_algorithm}.
  We present results using three different initial conditions.

  The first initial condition is a smooth sinusoidal function
  with one inflection point, $g_1:[0,\rho_0]\to\R^+$.
  For $r_1\in(0.5, 1)$ and $r_2\in\R^+$ the function $g_1$ is given by
  \begin{align*}
    B &= \frac{\pi}{2r_1}, &
    A &= \frac{\rho_0}{r_2 ( 1 + \abs{\cos(B)})},\\
    D &= -A \cos(B\rho_0), &
    g_1(u) &= A \cos\left(B u\right) + D.
  \end{align*}
  In our simulations we take $r_1=0.7$ and $r_2=2$. The parameter
  $r_2$ adjusts the length to heigh ratio of $g_1$. Physical measurements
  imply that $r_2\approx 2$. The parameter $r_1$ controls the position of
  the inflection point of $g_1$. We shall call this initial condition
  the ``inflection'' initial condition. Figure \ref{fig_initial_inflection}
  shows the graph of $g_1$ over $[0,3]$ for $r_1=0.7$ and $r_2=2$.

  The second initial condition is a compact perturbation of $g_1$ it is defined
  by
  \begin{align*}
    c &= \frac{2}{4}\rho_0, &
    r &= \frac{\rho_0 - c}{3},\\
    m &= 2\exp(1), &
    g_2(u) &= g_1(u) + m\exp\left(\frac{-r}{\max\{0, r - (u -c)^2\}}\right).
  \end{align*}
  The parameters $c,r,m$ describe, respectively, the centre, radius and
  height of a compact bump which is added to $g_1$. We shall call this initial
  condition the ``bump'' initial condition.
  Figure \ref{fig_initial_bump} shows the graph of $g_2$ over $[0,3]$.

  The third initial condition is a function estimated from a photograph of 
  experimental data which is then scaled and translated to give a function
  over the interval $[0,3]$. Figures \ref{subfigure_experiemnt}
  and \ref{subfigure_experiemnt_interp} present the experimental image 
  along with the estimates of the location of the dorsal closure (DC)
  leading edge at various times. Figure \ref{fig_initial_experimental}
  gives the graph of the experimental data.

  We use the \python{COFFEE},
  \cite{doulis2019coffee},
  package to provide the ``numerical infrastructure''
  for our simulations. 
  The \python{COFFEE} package was written in Python
  with ease of use as mind and has been used in numerous
  publications, e.g. \cite{frauendiener2021non,frauendiener2021can} see also the
  citations of \cite{doulis2019coffee}. 
  This paper is the first example of \python{COFFEE} being used to
  solve a parabolic system. 
  The \python{COFFEE} package
  has previously shown the ability to achieve round off. Thus we feel no need
  to demonstrate the correct implementation of algorithms by achieving round
  off in this paper. 
  The code was run on a commercial i5-10400 cpu with 16gb of ram running
  Ubuntu server 20.04.

  Simulations of the three initial conditions
  were performed on grids with $n=20 \times 2^i$ for $i=0, \ldots, 10$.
  All error calculations were performed with reference to the
  highest resolution simulation. 
  The $\log_2 L^\infty$ convergence rates
  at time $4$ are presented in
  Table \ref{table_error}.
  Due to the nature of the PDE the error
  results for other times are similar.

    \newcommand{\tablescale}{0.79}
    \begin{table}
      \hfill
      \scalebox{\tablescale}{
        \begin{tabular}{rlr}
          \multicolumn{3}{c}{Error for inflection initial data at time 4.0}\\
          \toprule
          $\delta u$ & $\ln L^\infty$ error & Rate of convergence\\
          \midrule
										0.15 & -2.324954    &              \\
                   0.075 & -2.902935    &    0.5780    \\
                  0.0375 & -3.677378    &    0.7744    \\
                 0.01875 & -4.564898    &    0.8875    \\
                0.009375 & -5.516336    &    0.9514    \\
               0.0046875 & -6.509211    &    0.9929    \\
              0.00234375 & -7.541358    &    1.0321    \\
             0.001171875 & -8.633302    &    1.0919    \\
						0.0005859375 & -9.851899    &    1.2186    \\
           0.00029296875 & -11.434980   &    1.5831   \\
          \bottomrule
        \end{tabular}
      }
      \hfill
      \scalebox{\tablescale}{
        \begin{tabular}{rlr}
          \multicolumn{3}{c}{Error for bump initial data at time 4.0}\\
          \toprule
          $\delta u$ & $\ln L^\infty$ error & Rate of convergence\\
          \midrule
                    0.15 & -1.483390    &           \\   
                   0.075 & -2.067566    &    0.5842 \\   
                  0.0375 & -2.848158    &    0.7806 \\   
                 0.01875 & -3.738997    &    0.8908 \\   
                0.009375 & -4.692272    &    0.9533 \\   
               0.0046875 & -5.686106    &    0.9938 \\   
              0.00234375 & -6.718740    &    1.0326 \\   
             0.001171875 & -7.810930    &    1.0922 \\   
            0.0005859375 & -9.029651    &    1.2187 \\   
           0.00029296875 & -10.612794   &    1.5831 \\
          \bottomrule
        \end{tabular}
      }\hfill \\[2em]
      \centering
      \scalebox{\tablescale}{
        \begin{tabular}{rlr}
          \multicolumn{3}{c}{Error for experimental initial data at time 4.0}\\
          \toprule
          $\delta u$ & $\ln L^\infty$ error & Rate of convergence\\
          \midrule
                    0.15 & -2.523499                \\  
                   0.075 & -3.136484    &    0.6130 \\   
                  0.0375 & -3.927557    &    0.7911 \\   
                 0.01875 & -4.823011    &    0.8955 \\   
                0.009375 & -5.778167    &    0.9552 \\   
               0.0046875 & -6.772835    &    0.9947 \\   
              0.00234375 & -7.805873    &    1.0330 \\   
             0.001171875 & -8.898258    &    1.0924 \\   
            0.0005859375 & -10.117087   &    1.2188 \\   
           0.00029296875 & -11.700276   &    1.5832 \\
          \bottomrule
        \end{tabular}
      }
      \caption{The rates of convergence of the $\ln L^\infty$ 
			error computed against
      the highest resolution simulation for each of
			the described initial data.
      }
      \label{table_error}
    \end{table}

\section*{Acknowledgements}
The first author was supported by an IPRS Scholarship at University of Wollongong.
She is grateful for their support.
The third athor acknowledges support from ARC DECRA DE190100379.

\bibliographystyle{plain}
\bibliography{mbib}

\end{document}